\documentclass[cmp]{svjour}  
\usepackage[centertags]{amsmath}
\usepackage{amsfonts}
\usepackage{amsmath}
\usepackage{amssymb}
\usepackage{tikz}
\journalname{Communications in Mathematical Physics}
\newcommand{\LL}{\mathcal{L}}

\newcommand{\RR}{{\mathbb R}}
\newcommand{\NN}{{\mathbb N}}
\newcommand{\ZZ}{{\mathbb Z}}
\newcommand{\CC}{{\mathbb C}}
\newcommand{\FF}{\mathcal{F}}
\newcommand{\KK}{\mathcal{K}}
\newcommand{\p}{\prime}
\newcommand{\ep}{\tilde{\varepsilon}}

\newcommand{\bigo}{\mathcal{O}}
\newcommand{\stooo}{IS^3\big(R(\log R)^{2k-1},\mathcal{Q}\big)}
\newcommand{\soooop}{IS^1\big(R(\log R)^{2k-1},\mathcal{Q}^\p\big)}
\newcommand{\sttoe}{IS^3\big(R^3(\log R)^{2k-1},\mathcal{Q}\big)}
\newcommand{\somoee}{IS^1\big(R^{-1}(\log R)^{2k},\mathcal{Q}\big)}
\newcommand{\soooep}{IS^1\big(R(\log R)^{2k-1},\mathcal{Q}^\p\big)}
\begin{document}
\textbf{Author:} Sohrab M. Shahshahani\\\\\\\\
\abstract{We construct a one parameter family of finite time blow ups to the co-rotational wave maps problem from $S^2\times \RR$ to $S^2,$ parameterized by $\nu\in(\frac{1}{2},1].$ The longitudinal function $u(t,\alpha)$ which is the main object of study will be obtained as a perturbation of a rescaled harmonic map of rotation index one from $\RR^2$ to $S^2.$ The domain of this harmonic map is identified with a neighborhood of the north pole in the domain $S^2$ via the exponential coordinates $(\alpha,\theta).$ In these coordinates $u(t,\alpha)=Q(\lambda(t)\alpha)+\mathcal{R}(t,\alpha),$ where $Q(r)=2\arctan{r},$ is the standard co-rotational harmonic map to the sphere, $\lambda(t)=t^{-1-\nu},$ and $\mathcal{R}(t,\alpha)$ is the error with local energy going to zero as $t\rightarrow 0.$ Blow up will occur at $(t,\alpha)=(0,0)$ due to energy concentration, and up to this point the solution will have regularity $H^{1+\nu-}.$

\def\thesection{\arabic{section}}
\section{Introduction}
\subsection{Set-up and History}
We study co-rotational wave maps $U:S^2\times\RR\rightarrow S^2\hookrightarrow\RR^3.$ Co-rotational means that the rotation index is one, or more precisely $U(t,\omega x)=\omega U(t,x)$ for $(t,x)\in\RR\times S^2,$ and $\omega\in SO(2),$ acting in the standard way on $S^2.$ The metric on the target sphere is induced from the ambient $\RR^3,$ and the domain is equipped with the standard metric of negative index $1,~m_{\beta\gamma}.$ With Einstein's summation convention in force, $\partial^\beta=m^{\beta\gamma}\partial_\gamma,$ $(m^{\beta\gamma})=(m_{\beta\gamma})^{-1}$, $d\sigma$ the associated volume element on the domain and $\langle\cdot,\cdot\rangle$ denoting the standard inner product on $\RR^3,$ wave maps are characterized by being critical with respect to the functional

\[U\mapsto \int_{S^2\times\RR}\langle\partial_\beta U,\partial^\beta U\rangle d\sigma,\quad\beta=0,1,2.\]

Recall that the energy

\[\mathcal{E}(U):=\int_{S^2}\langle DU(t,\cdot),DU(t,\cdot)\rangle d\mathrm{vol}_{S^2}\]

is conserved for such critical points. If we use spherical coordinates $\alpha,\theta$ on the domain sphere and let $u$ stand for the longitudinal angle on the target sphere, we can think of $U$ as $(t,\alpha,\theta)\mapsto(u(t,\alpha),\theta),$ where $u:(0,\pi)\rightarrow (0,\pi)$ is now a scalar map. $u$ then satisfies

\begin{equation}\label{wavemaps}
-u_{tt}+u_{\alpha\alpha}+\cot\alpha u_\alpha = \frac{\sin(2u)}{2\sin^2\alpha}.
\end{equation}

This equation will be the main object of study in this paper. The initial data will be given at time $t_0$ small enough, and we will solve the equation backwards in time. Note that the problem at hand is critical in the energy sense. Shatah and Tahvildar-Zadeh \cite{ST} proved the existence of stationary solutions to this problem under less restrictive spatial equivariance assumptions. Their results include well-posedness of the Cauchy problem in the small energy setting. Here we are interested in constructing blow ups through an energy concentration scenario. Our construction will closely follow that of Krieger, Schlag and Tataru \cite{KST} where blow ups are constructed for co-rotational wave maps from $\RR^2\times\RR$ to $S^2.$  The detailed discussion of the results in \cite{KST} is deferred to the next subsection. Here we briefly discuss some major recent developments, since the appearance of \cite{KST}. This is by no means an exhaustive account, although the reader can consult our references, in particular \cite{KST} and \cite{RR}, for a more comprehensive survey. Krieger et al. have used their method to construct blow ups for other wave type equations, in particular the NLW \cite{KST2} and Yang-Mills \cite{KST3} problems. In the setting of wave maps C\^{a}rstea, \cite{C} has extended the results of \cite{KST} to the case where the target is a more general surface of revolution. For the critical focusing wave equation, Donninger and Krieger \cite{DK} have recently modified these methods and constructed solutions which blow up at infinity. Going back to the wave map  problem, Rodnianski and Sterbenz \cite{RS} considered $k-$equivariant wave maps from $\RR^{2+1}$ to $S^2,$ for $k\geq4.$ Using different techniques from those of \cite{KST}, for each homotopy class $k,$ they construct finite time blow up solutions which remain smooth until the blow up time. Moreover, they demonstrate that their blow up construction is stable under small equivariant perturbations of the initial data. In \cite{RR} Raphael and Rodnianski study the problem for all rotation numbers $k\geq1$ (as well as the $SO(4)$ Yang-Mills), and are able to construct stable blow up solutions for all values of $k.$ In their construction, for $k\geq2$ the prescribed stable blow up rate is independent of the choice of initial data in the open set of data leading to blow up, but there is an additional degree of freedom depending on the initial data for $k=1.$ More precisely, if $T$ denotes the blow up time, for $k\geq2$ the blow up rate considered by Raphael and Rodnianski is
\[\lambda(t)=c_k(1+o(1))\frac{T-t}{|\log(T-t)|^{\frac{1}{2k-2}}},\]
where as for $k\geq1$ they have
\[\lambda(t)=(T-t)e^{-\sqrt{|\log(T-t)|}+O(1)}.\]
Related to this work, but in the context of Schrodinger equations, are the papers \cite{MR1}, \cite{MR2}, \cite{MR3}, \cite{MR4}, \cite{MR5}, \cite{R}, \cite{MRR1}, \cite{MRR2} by Merle, Raphael, and Rodnianski, and a similar result by Perelman \cite{P}. The construction of blow up wave maps by rescalings of harmonic maps are complemented by the result of Cote, Kenig, and Merle \cite{CKM}, where scattering is shown in the energy norm for the co-rotational case with initial data of energy smaller than the lowest energy harmonic map. See also \cite{Co} where the question of stability of harmonic map profiles is studied. We note that in the case where the dimension of the domain is larger than two, the study of singularities goes back to \cite{CST} and \cite{Sha}. As the focus here is the formation of singularities, we do not discuss the developments in the well-posedness and scattering directions in detail, but we refer the reader to the important papers \cite{KR}-\cite{KS},\cite{StT1}-\cite{StT2}, and \cite{T1}-\cite{Tat4}, and the references therein.\\\\
\subsection{General Outline}
After this short historical note we proceed to a general outline of the argument. As mentioned above, this outline will be very similar to that in \cite{KST}. It was shown by Struwe in \cite{S} that if a co-rotational wave map from $\RR^{2+1}$ to $S^2$ blows up, then a rescaled version of this map converges to a time-independent harmonic map from $\RR^2$ to $S^2$ at the blow up point. Motivated by this, Krieger et al. imposed the ansatz $u=Q+\mathcal{R}$ where $Q$ is a certain rescaling (to be specified below) of the harmonic map $Q(r)=2\arctan(r),$ and $\mathcal{R}$ is the remainder $u-Q.$ They then find $\mathcal{R}$ in such a way that $u$ becomes a wave map and the energy of $\mathcal{R}$ decays as $t\rightarrow 0,$ leading to energy concentration at the origin. The idea here will be similar. The difference is that $Q$ will no longer be a harmonic map from the whole space $S^2$ to the target $S^2,$ but rather from $\RR^2$ identified with a neighborhood of the north pole via spherical coordinates $(\alpha,\theta).$ One can think of this $\RR^2$ as the tangent space at the north pole, mapped to the sphere via the exponential map. More precisely we let

\[u(t,\alpha)=Q(\lambda(t)\alpha)+\mathcal{R}(\alpha,t),\quad\lambda(t)=t^{-1-\nu}.\]

Here $\nu\in(\frac{1}{2},1]$ is a fixed parameter. The reason for this choice of $\lambda$ is that according to \cite{S} if we have blow up in the flat case, then there is a decomposition of $u$ as above (with $r$ denoting the radial coordinate instead of $\alpha$) and a sequence of times $t_i\rightarrow0$ with $\lambda(t_i)|t_i|\rightarrow \infty$ such that the rescaled functions $u(t_i,\frac{r}{\lambda(t_i)})$ converge to $Q(r)$ in the strong energy topology. Note that even though this bubbling result by Struwe - which to the best of my knowledge is not available in the case of a curved background - served as a motivation in \cite{KST}, its validity was not used explicitly in the construction.\\

The idea now is to rewrite (\ref{wavemaps}) as an equation for $\mathcal{R}$ and use a fixed point argument relative to an appropriate norm to find $\mathcal{R}.$ Because we want the energy of $\mathcal{R}$ to decay as time goes to zero, our iteration norms will involve growth in time. However, the error coming from the failure of the rescaled harmonic map $Q(\lambda(t)\alpha)$ to be an exact wave map, is too large to be placed in these iteration spaces (at least the ones used in \cite{KST} and this paper). For this reason we have to further breakdown the error $\mathcal{R}$ as $Q(\lambda(t)\alpha)+\mathcal{R}(t\alpha)=u_{2k-1}(t,\alpha)+\ep(t,\alpha),$ where $u_{2k-1}$ is an approximate solution to the wave maps equation, and $\ep$ is the error playing the role that $\mathcal{R}$ played before. The index $2k-1$ indicates that the approximate solution will be refined iteratively, starting with $u_0(t,\alpha)=Q(\lambda(t)\alpha).$ The point is that if $k$ is sufficiently large, the error due to $u_{2k-1}$ not being an exact wave map is small enough to be placed in our iteration spaces. Of course we will need to make sure that the energy of $u_{2k-1}-Q$ decays as we approach the blow up time. Before we state the main theorem we remind the reader that in the equivariant setting the energy is given,up to a constant, by

\[\mathcal{E}(u)=\int_0^\pi\Big[\frac{1}{2}(u_t^2+u_\alpha^2)+\frac{\sin^2(u)}{2\sin^2\alpha}\Big]\sin\alpha d\alpha,\]

and the local energy relative to the origin is defined as

\[\mathcal{E}_{\mathrm{loc}}(u)=\int_0^t\Big[\frac{1}{2}(u_t^2+u_\alpha^2)+\frac{\sin^2(u)}{2\sin^2\alpha}\Big]\sin\alpha d\alpha.\]

\begin{theorem}\label{BlowUp}
Let $\nu\in(\frac{1}{2},1]$ be arbitrary and $t_0>0$ sufficiently small. Define $\lambda(t)=t^{-1-\nu}$ and fix a large integer $N.$ Then there is a function $u^e$ satisfying \footnote{Following \cite{KST} we call $u^e$ the elliptic modifier. Also for noninteger $\beta,$ $C^\beta$ means $C^{[\beta],\beta-[\beta]}.$}

\begin{align*}
&u^e\in C^{\nu+\frac{1}{2}-}(\{0<t< t_0,~r\leq t\}),\\
&\mathcal{E}_{\mathrm{loc}}(u^e)(t)\lesssim (t(\lambda(t)))^{-2}|\log t|^2,\quad \mathrm{as}~t\rightarrow 0,
\end{align*}

and a blow up solution $u$ to (\ref{wavemaps}) in $(0,t_0]$ which is of the form

\[u(t,\alpha)=Q(\lambda(t)\alpha)+u^e(t,\alpha)+\ep(t,\alpha),\quad 0\leq \alpha\leq t,\]

where $\ep$ decays as $t\rightarrow 0;$ more precisely,

\[\ep\in t^NH^{1+\nu-}(S^2),\quad \ep_t\in t^{N-1}H^{\nu-}(S^2),\quad \mathcal{E}_{\mathrm{loc}}(\ep)(t)\lesssim t^N,\quad \mathrm{as}~t\rightarrow0,\]

with spatial norms which are uniformly controlled as $t\rightarrow0.$ The solution $u(t,\alpha)$ extends as an $H^{1+\nu-}$ solution to all of $(0,t_0]\times S^2,$ and its energy concentrates in the cuspidal region $0\leq r\lesssim \frac{1}{\lambda(t)}$ leading to blow up at $r=t=0.$

\end{theorem}
\subsection{Overview of the Proof}
In this subsection we provide a more detailed account of the blow up construction and the proof of theorem \ref{BlowUp}. The proof here will combine the method of proof of the corresponding theorem in \cite{KST} with a perturbative argument. By this we mean that equation (\ref{wavemaps}) is replaced by its flat analogue

\[-u_{tt}+u_{\alpha\alpha}+\frac{1}{\alpha}u_\alpha-\frac{\sin(2u)}{2\alpha^2},\]

and the difference is included in the source term. In the appendix, we have collected the results from \cite{KST} which can be used directly in our setting.\\\\

Section 2 is devoted to constructing the approximate solution which we refer to as the elliptic modifier. The construction consists of adding successive refinements to the previous approximate solutions, starting with $u_0(t,\alpha)=Q(t,\alpha).$ This is done in two steps: first near $\alpha=0$ and then near the boundary of the cone $\alpha=t.$ The error at step $2k-1$

\[e_{2k-1}:=(-\partial_t^2+\partial_\alpha^2+\cot\alpha\partial_\alpha)u_{2k-1}-\frac{\sin(2u_{2k-1})}{2\sin^2\alpha},\]

will satisfy

\[e_{2k-1}=\bigo\Big(\frac{R(\log(2+R))^{2k-1}}{t^{-2k+2}t^2(t\lambda)^2}\Big).\]

Here $R=\lambda\alpha$ is the rescaled radial variable, and the $\bigo(\cdot)$ terms are uniform in $r\leq t,~0<t<t_0.$ Note that for $\nu>1$ the decay in time is slower than the decay of the corresponding error $\tilde{e}_{2k-1}$ in the flat case from \cite{KST} which satisfies

\[\tilde{e}_{2k-1}=\bigo\Big(\frac{R(\log(2+R))^{2k-1}}{t^2(t\lambda)^{2k}}\Big).\]

Even though we restrict to the case $\nu\leq1,$ the renormalization step is carried out for all values of $\nu>\frac{1}{2}.$ This slower decay is the result of the error coming from replacing the curved wave equation by its flat counterpart, and is observed in the analysis near the boundary of the cone. It entails that if we could take $\nu>1,$ to be able to close the final fixed point argument we would need more iterations of our approximation scheme (see section 5). In our analysis we will use the variable $R$ near $\alpha=0$ and the self-similar variable $a=\frac{\alpha}{t}$ near the boundary of the cone $a=1.$ The approximate solution $u_{2k-1}$ and the error $e_{2k-1}$ will involve expressions of the form $(1-a)^{\nu-\frac{1}{2}}\log^m(1-a),$ which dictate the upper bound on the regularity of our final solution. In both cases we will encounter a Sturm-Liouville ODE which is why we use the name "elliptic modifier." Even though the arguments in this section are similar to those in \cite{KST} we are not  able to use most of the results from that paper directly. In particular to be able to track the error from our perturbative argument we will need to introduce extra auxiliary variables in our function spaces (see section 2 for precise definitions), and we will use a slightly watered down version of the function algebras $\mathcal{Q}$ and $\mathcal{Q}^\p$ (in particular we do not track the exact number of $\log(1-a)$ factors at each iteration, and hence we dispense with the index on the algebras $\mathcal{Q}_k$ and $\mathcal{Q}^\p_k$).\\\\

Section 3 corresponds to sections 4 and 7 in \cite{KST} and does not contain new results. We make the ansatz $u(t,\alpha)=u_{2k-1}(t,\alpha)+\varepsilon(t,\alpha)$ and recast equation (\ref{wavemaps}) as an equation for $\epsilon$ as

\begin{align*}
-\varepsilon_{tt}+\varepsilon_{\alpha\alpha}+\frac{1}{\alpha}\varepsilon_\alpha-\frac{\cos(2Q(\lambda\alpha))}{\alpha^2}\varepsilon=N_{2k-1}(\varepsilon)+H_1(\varepsilon)+e_{2k-1}.
\end{align*}

Here $N_{2k-1}$ contains the nonlinearity coming from the original equation on the curved background and is defined in (\ref{Nodd}), and $H_1,$ which is responsible for the perturbative error coming from replacing the curved equation with the flat one, is defined in (\ref{Hodd}). We will change variables to $R=\lambda\alpha$ and $\tau=\frac{-1}{\nu}t^{-\nu}$ (so now the blow up time corresponds to $\tau=\infty$) and rewrite the equation above in terms of the new variable $\ep(\tau,R)=R^{\frac{1}{2}}\varepsilon(t(\tau),\alpha(\tau,R))$ as

\begin{align*}
\Big(-\Big(\partial_\tau+\frac{\lambda_\tau}{\lambda}R\partial_R\Big)^2&+\frac{1}{4}(\frac{\lambda_\tau}{\lambda})^2+\frac{1}{2}\partial_\tau(\frac{\lambda_\tau}{\lambda})\Big)\ep-\LL\ep\\
&=\lambda^{-2}R^{\frac{1}{2}}(N_{2k-1}(R^{-\frac{1}{2}}\ep)+H_1(R^{-\frac{1}{2}}\ep)+e_{2k-1}),
\end{align*}

where

\[\LL:=-\partial_R^2+\frac{3}{4R^2}-\frac{8}{(1+R^2)^2}.\]

$\LL$ is a self adjoint operator on $L^2((0,\infty),dR)$ (see the appendix for the definition of the domain of $\LL$) and its spectral properties, which are analyzed in \cite{G} and \cite{KST} are summarized in the appendix. In particular with $\rho(\xi)d\xi$ denoting the spectral measure associated with $\LL$ we consider the distorted Fourier transform

\[\FF:f\mapsto \hat{f}(\xi)=\lim_{b\rightarrow\infty}\int_0^b\phi(R,\xi)f(R)dR,\]

with inverse

\[\FF^{-1}:\hat{f}\mapsto f(r)=\lim_{\mu\rightarrow\infty}\int_0^\mu\phi(R,\xi)\hat{f}(\xi)\rho(\xi)d\xi.\]

We apply $\mathcal{F}$ to the equation satisfied by $\ep$ and let $x=\mathcal{F}\ep$ to get

\begin{align}
-\Big(\partial_\tau-2\frac{\lambda_\tau}{\lambda}\partial_\xi\Big)^2&x-\xi x\nonumber\\
                                                                    &=2\frac{\lambda_\tau}{\lambda}\KK\Big(\partial_\tau-2\frac{\lambda_\tau}{\lambda}\xi\partial_\xi\Big)x+\frac{\lambda_\tau^2}{\lambda^2}\Big(\KK^2-\frac{\nu}{1+\nu}\KK+2[\KK,\xi\partial_\xi]\Big)x\nonumber\\
                                                                    &~~-\Big(\frac{1}{4}\Big(\frac{\lambda_\tau}{\lambda}\Big)^2+\frac{1}{2}\partial_\tau\Big(\frac{\lambda_\tau}{\lambda}\Big)\Big)x\nonumber\\
                                                                    &~~+\lambda^{-2}\FF R^{\frac{1}{2}}\Big(N_{2k-1}(R^{-\frac{1}{2}}\FF^{-1}x)+H_1((R^{-\frac{1}{2}}\FF^{-1}x))+e_{2k-1}\Big).\nonumber
\end{align}

Here the operator $\KK$ is defined as

\[\widehat{R\partial_R u}=-2\xi\partial_\xi \hat{u}+ \KK\hat{u},\]

and as such encodes the deviation of $R\partial_R$ from being diagonalizable in the Fourier basis. We denote by $H$ the fundamental solution operator of the operator on the left hand side of the equation for $x$. The mapping properties of $\KK$ and $H$ are summarized in the appendix.\\\\

In section 4 we introduce the Sobolev norms associated with the operator $\LL.$ The final fixed point argument is carried out on the spaces defined by these norms. On the frequency side we define the weighted $L^2$ norms

\begin{align*}
\|f\|_{L^{2,s}_\rho}:=\Big(\int_0^\infty|f(\xi)|^2\langle\xi\rangle^{2s}\rho(\xi)d\xi\Big)^{\frac{1}{2}}.
\end{align*}

For functions of the spacial variable $R$ we define

\begin{align*}
\|u\|_{H^s_\rho}:=\|\hat{u}\|_{L^{2,s}_\rho}.
\end{align*}

To control the decay in time we introduce the $L^{N,\infty}L^{2,s}_\rho$ spaces with norm

\begin{align*}
\|f\|_{L^{\infty,N}L^{2,s}_\rho}:=\sup_{\tau\geq1}\tau^N\|f\|_{L^{2,s}_\rho}.
\end{align*}

$L^{N,\infty}H^s_\rho$ is defined similarly. The norm of the $N_{2k-1}$ term arising in the equation satisfied by the Fourier coefficient $x$ is controlled in \cite{KST}, and here is recorded in the appendix. The new ingredient in the nonlinearity is $H_1,$ and it is this nonlinearity which is responsible for the upper bound $\nu\leq1.$ In particular, in the renormalization step (Section 2) we carry out our computations for general $\nu>\frac{1}{2}.$\\\\

Finally in section 5 we assemble the results from the previous sections to find the Fourier coefficient $x=\mathcal{F}\ep$ using a contraction mapping argument. We will use certain $L^{\infty,N}L^{2,s}_\rho$ norms for this fixed point argument, and we need to take $k$ large enough depending on $N$ to be able to place the term involving $e_{2k-1}$ in our iteration space. We will then retrace our steps back to find $\ep$ and verify that it satisfies the appropriate energy decay as $t\rightarrow0.$

\subsection{Acknowledgments.}
I would like to thank my advisor Professor Joachim Krieger for suggesting this problem and for his guidance and encouragement throughout the work. I would also like to thank the post-docs in our group Roland Donninger, Joules Nahas, and Willie Wong for helpful discussions. 
\section{The Elliptic Profile Modifier}
%
%
We will construct approximate solutions to the wave maps equation in the light cone $\mathcal{C}_0$ at the origin

\begin{align*}
\mathcal{C}_0:=\{(t,\alpha)~:~0\leq \alpha < t,~0<t<t_0 \}.
\end{align*}

The approximations will be perturbations of a time-dependent harmonic map profile

\[u_0(t,\alpha)=Q(\lambda(t)\alpha),\quad \lambda(t)=t^{-1-\nu}.\]

We will closely follow the construction in \cite{KST}, while making the necessary modifications necessary to handle the curved metric on the background. To begin, we introduce the following set of auxiliary variables:


\begin{align*}
&b_1=\frac{(\log(2+R^2))^2}{(t\lambda)^2}=t^{2\nu}(\log(2+R^2))^2,\\
&b_2=\frac{(\log(2+R^2))^2}{(t^\nu\lambda)^2}=t^{2}(\log(2+R^2))^2,\\
&b_3=\alpha^2,\\
&b_4=(t\lambda)^{-2}=t^{2\nu}\\
&b_5=(t^\nu\lambda)^{-2}=t^2\\
&R=\lambda\alpha,\\
&a=\frac{\alpha}{t}.
\end{align*}


Technically speaking, we do not need $b_3$ because $b_3=a^2b_5.$ However, as even analytic functions of $\alpha$ arise naturally in our computations, we prefer to keep this as a separate variable. Since we are working in the light cone $\mathcal{C}_0$, the variables $b_i$ will be restricted to small intervals $[0,b_{0,i}],$ respectively, and $a$ will take values in $[0,1].$ Let

\begin{align*}
\Omega:=[0,1]\times[0,\infty)\times[0,b_{0,1}]\times\cdots\times[0,b_{0,5}].
\end{align*}

and denote the projection of $\Omega$ onto the last six factors by $\Omega_a,$ and the projection onto the last five factors by $\Omega_{a,R}.$ We will work with asymptotic expansions of the functions we are concerned with in order to study their growth and decay. To be able to make optimal use of the new set of variables we need a few definitions. These definitions, and the statement of theorem \ref{elliptic modifier} are all that will be used from this section in the remainder of the paper. The reader may therefore wish to skip the proof of the theorem in the first reading. The definitions below are technical. For the sake of readability the reader may think of $IS^m\big(R^k(\log R)^l,\mathcal{Q}\big)$ as the space of functions satisfying the following properties. They are analytic near $R=0$ and vanish to order $m$ there. For large $R$ they behave roughly like $R^k(\log R)^m.$ $\mathcal{Q}$ means that near the boundary of the cone $a=1,$ the functions may contain singularities of the form $(1-a)^{\nu+\frac{1}{2}}(\log(1-a))^n$ for some integer $n.$ If $\mathcal{Q}$ is replaced by $\mathcal{Q}^\p,$ the exponent $\nu+\frac{1}{2}$ should be replaced by $\nu-\frac{1}{2}.$ The auxiliary variables $b_i$ are harmless and may be safely ignored. We now provide the precise definitions.\\\\

\begin{definition}\label{Qdef}
$\mathcal{Q}$ is the algebra of continuous functions $q:[0,1]\rightarrow \RR$ with the following properties:
\begin{description}
\item[]  $~~$ (i) q is analytic in $[0,1)$ with an even expansion at $0.$
\item[]  $~~$ (ii) Near $a=1$ we have an absolutely convergent expansion of the form
    \begin{align*}
    q=q_0(a) +\sum_{i=1}^{i=\infty}\Big(&(1-a)^{(2i-1)\nu+\frac{1}{2}}\sum_{j=0}^{\infty}q_{2i-1,j}(a)(\log(1-a))^j\\
                                   &(1-a)^{2i\nu+1}\sum_{j=0}^{\infty}q_{2i,j}(a)(\log(1-a))^j\Big)
    \end{align*}
    with analytic coefficients $q_0,~q_{i,j},$ such that for each $i$ only finitely many of the $q_{i,j}$ are not identically equal to zero.
\end{description}
\end{definition}

\begin{definition}
$\mathcal{Q}^\p$ is the space of continuous functions $q:[0,1]\rightarrow \RR$ with the following properties:
\begin{description}
\item[]  $~~$ (i) q is analytic in $[0,1)$ with an even expansion at $0.$
\item[]  $~~$ (ii) Near $a=1$ we have an absolutely convergent expansion of the form
    \begin{align*}
    q=q_0(a) +\sum_{i=1}^{i=\infty}\Big(&(1-a)^{(2i-1)\nu-\frac{1}{2}}\sum_{j=0}^{\infty}q_{2i-1,j}(a)(\log(1-a))^j\\
                                   &(1-a)^{2i\nu}+\sum_{j=0}^{\infty}q_{2i,j}(a)(\log(1-a))^j\Big)
    \end{align*}
    with analytic coefficients $q_0,~q_{i,j},$ such that for each $i$ only finitely many of the $q_{i,j}$ are not identically equal to zero.
\end{description}
\end{definition}

\begin{definition}
$S^m(R^k(\log R)^l)$ is the class of analytic functions $v:[0,\infty)\rightarrow \RR$ with the following properties:
\begin{description}
\item[]$~~$ (i) $v$ vanishes of order $m$ at $R=0$ and $v(R)=R^m\sum_{j=0}^{j=\infty}c_jR^{2j}$ for small $R.$
\item[]$~~$ (ii) $v$ has a convergent expansion near $R=\infty,$
    \begin{align*}
    v=\sum_{0\leq j\leq l+i}c_{ij}R^{k-2i}(\log R)^j
    \end{align*}
\end{description}
\end{definition}

Finally,

%
\begin{definition}
a) $S^m(R^k(\log R)^l,\mathcal{Q})$ is the class of analytic functions $v:\Omega\rightarrow \RR$ so that
\begin{description}
\item[]$~~$(i) $v$ is analytic as a function of $R,b_i$
                \[v:\Omega_a\rightarrow\mathcal{Q}\]
\item[]$~~$(ii) $v$ vanishes to order $m$ at $R=0$ and is of the form
                \[v\approx R^m\sum_{j=0}^{j=\infty}c_j(a,b_1,b_2,b_3,b_4,b_5)R^{2j}\]
            around $R=0.$
\item[]$~~$(iii) $v$ has a convergent expansion at $R=\infty,$
            \begin{align*}
            v(R,\cdot,b_1,b_2,b_3,b_4,b_5)=\sum_{0\leq j\leq l+1}c_{ij}(\cdot,b_1,b_2,b_3,b_4,b_5)R^{k-2i}(\log R)^j
            \end{align*}
         where the coefficients $c_{ij}:\Omega_{a,R}\rightarrow \mathcal{Q}$ are analytic with respect to $b_i.$
\end{description}

b) $IS^m(R^k(\log R)^l, \mathcal{Q})$ is the class of analytic functions $w$ on the cone $\mathcal{C}_0$ which can be represented as
    \[w(t,\alpha)=v(R,a,b_1,b_2,b_3,b_4,b_5),\quad v\in S^m(R^k(\log R)^l,\mathcal{Q}).\]
\end{definition}

Note that this last representation is in general not unique. Note also that these definitions are slightly different from the corresponding ones in \cite{KST}.\\\\
Our goal in this section is to prove the following theorem.

\begin{theorem}\label{elliptic modifier}
For any $k\in\NN,$ there exists an approximate solution

\[u_{2k-1}\in Q(R)+\frac{1}{(t\lambda)^2}IS^3\big(R\log R,\,\mathcal{Q}\big)\]

to the wave maps equation (\ref{wavemaps}), so that the corresponding error $e_{2k-1}$ satisfies

\[t^2e_{2k-1}\in\sum_{j=0}^{k-1}\frac{1}{(t^\nu\lambda)^{2j}(t\lambda)^{2(k-j)}}\soooop .\]

\end{theorem}

It is worth mentioning that for $\nu>1,$ the decay in time of the error here ($t^{2k+2\nu-4}$) is slower than that of the corresponding error in \cite{KST} ($t^{2k\nu-2}$). 

\begin{proof}
Our strategy will be the same as that in \cite{KST}. We start with $u_0$ as a first approximation and improve our approximations inductively by adding a the correction $v_k$ at the $k$'th step:

\[u_k=u_{k-1}+v_k.\]

The error at this step is

\[-e_k=\Big(-\partial_t^2+\partial^2\alpha+\cot\alpha\partial_\alpha\Big)u_k-\frac{\sin(2u_k)}{2\sin^2\alpha}.\]

If $u$ were an exact solution, then the difference $\varepsilon=u-u_{k-1}$ would satisfy

\begin{align}
\Box_g \varepsilon&-\frac{\cos(2u_{k-1})}{2\sin^2\alpha}\sin(2\varepsilon)\nonumber\\
                &+ \frac{\sin(2u_{k-1})}{2\sin^2\alpha}(1-\cos(2\varepsilon))=e_{k-1}, \label{errorequation}
\end{align}

where $\Box_g=-\partial_t^2+\partial^2\alpha+\cot\alpha\partial_\alpha.$ As in \cite{KST} we linearize this equation around $\varepsilon=0,$ and substitute $u_0$ for $u_{k-1}$ to get

\begin{equation}\label{errorlinequation}
\Big(\Box_g-\frac{\cos(2u_0)}{2\sin^2\alpha}\Big)\varepsilon\approx e_{k-1}.
\end{equation}

For $r\ll t$ the time derivative should play a less important role and we approximate the equation above by

\[\Big(\partial_\alpha^2+\cot\alpha\partial_\alpha-\frac{\cos(2u_0)}{\sin^2\alpha}\Big)\varepsilon=e_{k-1},\]

In fact we make one further simplification and replace $\sin\alpha$ by $\alpha$ and $\cos\alpha$ by $1:$

\[\Big(\partial_\alpha^2+\frac{1}{\alpha}\partial_\alpha-\frac{\cos(2u_0)}{2\alpha^2}\Big)\varepsilon=e_{k-1}.\]

It is understood that the error $e_{k-1}$ includes the one previously defined and the new error resulting from this last simplification (this will be written out completely below). This simplification allows us to use the results from \cite{KST} more directly. Similarly for $r\approx t,$ where we will later use the self similar variable $a,$ we approximate $\cos(2u_0)$ by $1$ and set $\sin\alpha\approx\alpha,~\cos\alpha\approx 1,$ and rewrite (\ref{errorlinequation}) as

\[\Big(-\partial_t^2+\partial_\alpha^2+\frac{1}{\alpha}\partial_\alpha-\frac{1}{\alpha^2}\Big)\varepsilon=e_{k-1}.\]

The heuristic arguments above lead us to the following two step improvement of our approximations. First we consider the region $r\ll t$ and define $v_{2k+1}$ as a solution of

\begin{equation}\label{voddeq}
\Big(\partial_\alpha^2+\frac{1}{\alpha}\partial_\alpha-\frac{\cos(2u_0)}{\alpha^2}\Big)v_{2k+1}=e_{2k}^0,
\end{equation}

with zero initial data at $\alpha=0$ (to be interpreted suitably below, as the coefficients are singular at $r=0$). Here $e_{2k}^0$ is the "large" part of the error $e_{2k}$ and will be defined during the iteration. In the second step we improve our approximation by adding $v_{2k+2}$ defined by

\begin{equation}\label{veveneq}
 \Big(-\partial_t^2+\partial_\alpha^2+\frac{1}{\alpha}\partial_\alpha-\frac{1}{\alpha^2}\Big)v_{2k+2}=e^0_{2k+1}.
\end{equation}

Again the zero initial data at $r=0$ needs proper interpretation, and $e_{2k+1}^0$ is the principal part of the error $e_{2k+1},$ to be defined below. Comparing with (\ref{errorequation}) we compute the successive errors as

\begin{align*}
&e_{2k}=e^1_{2k-1}+N_{2k}(v_{2k})+H_0(v_{2k}),\\
&e_{2k+1}=e^1_{2k}-\partial^2_tv_{2k+1}+N_{2k+1}(v_{2k+1})+H_1(v_{2k+1}).
\end{align*}

Here

\begin{align}
N_{2k+1}(v)=&\frac{\cos(2u_{2k})-\cos(2u_0)}{\sin^2\alpha}v+\frac{\sin(2u_{2k})}{2\sin^2\alpha}(\cos(2v)-1)\nonumber\\
            &+ \frac{\cos(2u_{2k})}{2\sin^2\alpha}(\sin(2v)-2v), \label{Nodd}
\end{align}

respectively

\begin{align}
N_{2k}(v)=&\frac{\cos(2u_{2k-1})-1}{\sin^2\alpha}v+\frac{\sin(2u_{2k-1})}{2\sin^2\alpha}(\cos(2v)-1)\nonumber\\
          & + \frac{\cos(2u_{2k-1})}{2\sin^2\alpha}(\sin(2v)-2v), \label{Neven}
\end{align}

and

\begin{equation}\label{Heven}
H_0(v)=\Big(\frac{1}{\alpha}-\cot\alpha\Big)\partial_\alpha v+\Big(\frac{1}{\alpha^2}-\frac{1}{\sin^2\alpha}\Big)v,
\end{equation}

respectively

\begin{equation}\label{Hodd}
H_1(v)=\Big(\frac{1}{\alpha}-\cot\alpha\Big)\partial_\alpha v+\Big(\frac{\cos(2u_0)}{\alpha^2}-\frac{\cos(2u_0)}{\sin^2\alpha}\Big)v.
\end{equation}

To carry out the construction just outlined and prove the theorem, we implement the following four step induction scheme. For each $k\geq1,$
\begin{align}
&v_{2k-1}\in\sum_{j=0}^{k-1}\frac{1}{(t^\nu\lambda)^{2j}(t\lambda)^{2(k-j)}}\stooo \label{vodd}\\
&t^2e_{2k-1}\in\sum_{j=0}^{k-1}\frac{1}{(t^\nu\lambda)^{2j}(t\lambda)^{2(k-j)}}\soooop \label{eodd}\\
&v_{2k}\in\sum_{j=0}^{k-1}\frac{1}{(t^\nu\lambda)^{2j}(t\lambda)^{2(k+1-j)}}\sttoe \label{veven}\\
&t^2e_{2k}\in\sum_{j=0}^{k-1}\frac{1}{(t^\nu\lambda)^{2j}(t\lambda)^{2(k-j)}}[\somoee \nonumber\\
&\quad\quad\quad\quad\quad\quad+\sum_{i=1}^5b_i\soooep]\label{eeven}.
\end{align}

Step $0:~t^2e_0 \in IS^1(R^{-1}).$
\begin{align*}
e_0&=\partial_t^2u_0-\partial_\alpha^2u_0-\frac{1}{\alpha}\partial_\alpha u_0+\frac{\sin (2u_0)}{2\alpha^2}\\
   &+(\frac{1}{\alpha}-\cot\alpha)\partial_\alpha u_0+(\frac{1}{2\sin^2\alpha}-\frac{1}{2\alpha^2})\sin(2u_0)\\
   &=\partial_t^2(Q(\lambda\alpha))+(\frac{1}{\alpha}-\cot\alpha)\partial_\alpha (Q(\lambda\alpha))+(\frac{1}{2\sin^2\alpha}-\frac{1}{2\alpha^2})\sin(2Q(\lambda\alpha)).
\end{align*}

So,

\begin{align*}
t^2e_0&=\Bigg[(1+\nu)^2\frac{4R}{(1+R^2)^2}-\nu(1+\nu)\frac{2R}{1+R^2}\Bigg]\\
      &+t^2\Bigg[\bigg(\frac{2(\sin^2\alpha-\alpha^2)}{\alpha^2\sin^2\alpha}\bigg)\bigg(\frac{R(1-R^2)}{(1+R^2)^2}\bigg)+\bigg(\frac{\sin\alpha-\alpha\cos\alpha}{\alpha^2\sin\alpha}\bigg)\frac{R}{1+R^2}\Bigg]
\end{align*}

Note that the functions of $\alpha$ in the second term are analytic and even, and that $t^2=b_5$. It follows that

\begin{align*}
t^2e_0\in IS^1(R^{-1}),
\end{align*}

as desired.\\

Step $1:$ define $v_{2k-1}$ such that (\ref{vodd}) holds.\\

Define the principal part $e^0_{2k-2}$ of $e_{2k-2}$ by setting $b_i=0,~i=1\cdots 5,$ in (\ref{eeven}) with $k$ replaced by $k-1.$ We can write $e_{2k-2}^0=\sum_{j=0}^{k-2}e_{2k-2,j}^0,$ with

\begin{align*}
t^2e_{2k-2,j}^0\in\frac{1}{(t^\nu\lambda)^{2j}(t\lambda)^{2(k-1-j)}}IS^1(R^{-1}(\log R)^{2k-2},Q).
\end{align*}

The non principal part $t^2e^1_{2k-2}=t^2(e_{2k-2}-e^0_{2k-2})$ belongs to

 \begin{align*}
 &\sum_{j=0}^{k-2}\sum_{i=1}^5\frac{b_i\big[IS^1(R^{-1}(\log R)^{2k-2},\mathcal{Q})+IS^1(R(\log R)^{2k-3},\mathcal{Q}^\p)\big]}{(t^\nu\lambda)^{2j}(t\lambda)^{2(k-1-j)}}\\
 &\subseteq \sum_{j=0}^{k-1} \frac{1}{(t^\nu\lambda)^{2j}(t\lambda)^{2(k-j)}}IS^1(R(\log R)^{2k-1},\mathcal{Q}^\p),
 \end{align*}

 which can be included in $e_{2k-1}.$ The inclusion above is clear for $i\neq3,$ and follows by writing $b_3=a^2/(t^\nu\lambda)^2$ for $i=3.$\\

 Now as in (\ref{voddeq}), the idea is to define $v_{2k-1,j}$ by requiring that

\[(\partial_\alpha^2+\frac{1}{\alpha}\partial_\alpha-\frac{\cos(2u_0)}{\alpha^2})v_{2k-1,j}=e^0_{2k-2,j}.\]

Note that if $v(R)$ is a function of only $R$ then $(\partial_\alpha^2+\frac{1}{\alpha}\partial_\alpha-\frac{\cos(2u_0)}{\alpha^2})v=\lambda^2Lv,$ where

\begin{equation}\label{Lelliptic}
L=\partial_R^2+\frac{1}{R}\partial_R-\frac{\cos(2u_0)}{R^2}=\partial_R^2+\frac{1}{R}\partial_R-\frac{1-6R^2+R^4}{R^2(1+R^2)^2}.
\end{equation}

Using this as a motivation we define $v_{2k-1,j}$ by requiring that

\[(t\lambda)^2Lv_{2k-1,j}=t^2e^0_{2k-2,j}.\]

In doing so, we view the equation above as an ODE in $R$ and treat all other variables as parameters. The errors resulting from this simplification will be studied in the next step. According to lemma \ref{lemma3.7} from the appendix $v_{2k-1,j}\in \frac{1}{(t^\nu\lambda)^{2j}(t\lambda)^{2(k-j)}}IS^3(R(\log R)^{2k-1},\mathcal{Q}).$ Therefore if we define

\[v_{2k-1}:=\sum_{j=0}^{k-2} v_{2k-1,j},\]

we have $v_{2k-1}\in\sum_{j=0}^{k-1} \frac{1}{(t^\nu\lambda)^{2j}(t\lambda)^{2(k-j)}}IS^3(R(\log R)^{2k-1},\mathcal{Q}).$\footnote{In fact we have $v_{2k-1}\in\sum_{j=0}^{k-2} \frac{1}{(t^\nu\lambda)^{2j}(t\lambda)^{2(k-j)}}IS^3(R(\log R)^{2k-1},\mathcal{Q}),$ but we do not use this.}\\

Step $2:$ show that with $v_{2k-1}$ defined as in the previous step, (\ref{eodd}) holds.\\

Thinking of $v_{2k-1}$ as a function of $t,~R,$ and $a$ we write

\[e_{2k-1}=e^1_{2k-2}+N_{2k-1}(v_{2k-1})+H_1(v_{2k-1})+E^tv_{2k-1}+E^av_{2k-1}.\]

$e^1_{2k-2}$ was taken care of in the previous step. $N_{2k-1}(v_{2k-1})$ and $H_1(v_{2k-1})$ account for the contributions from the nonlinearity and approximating $\sin\alpha$ by $\alpha.$ They are given by (\ref{Nodd}) and (\ref{Hodd}) respectively. $E^tv_{2k-1}$ contains the terms in $\partial_t^2v_{2k-1}(t,R,a)$ where no derivatives apply to $a$ and $E^av_{2k-1}$ the terms in

\[\Big(-\partial_t^2+\partial_r^2+\frac{1}{r}\partial_r\Big)v_{2k-1}(t,R,a)\]

where at least one derivative applies to $a.$ The only new contribution compared with \cite{KST} is $H_1(v_{2k-1}),$ which we now bound.

\begin{align*}
\Big(\frac{\cos(2u_0)}{\alpha^2}-\frac{\cos(2u_0)}{\sin^2\alpha}\Big)v_{2k-1}&=\frac{\sin^2\alpha-\alpha^2}{\alpha^2\sin^2\alpha}\cos(2u_0)v_{2k-1}\\
                                                                             &\in\sum_{j=0}^{k-1}\frac{1}{(t^\nu\lambda)^{2j}(t\lambda)^{2(k-j)}}\soooop,
\end{align*}

because the coefficient is an even analytic function of $\alpha$ and $\cos(2u_0)\in S^0(1).$ For the other term in $H_1$ we write

\begin{equation}\label{voddaltrep}
v_{2k-1}=\sum_{j=0}^{k-1}\frac{W_j(a,R)}{(t^\nu\lambda)^{2j}(t\lambda)^{2(k-j)}},\quad W_j\in \stooo,
\end{equation}

and compute

\begin{align*}
\Big(\cot\alpha-\frac{1}{\alpha}\Big)\partial_\alpha v_{2k-1}=\big(\frac{\alpha\cos\alpha-\sin\alpha}{\alpha^2\sin\alpha}\big)\sum_{j=0}^{k-1}\frac{a\partial_a W_j(a,R)+R\partial_R W_j(a,R)}{(t^\nu\lambda)^{2j}(t\lambda)^{2(k-j)}}.
\end{align*}

Since the coefficient is an even analytic function of $\alpha$ and since

\[a\partial_a:\mathcal{Q}\rightarrow \mathcal{Q}^\p,\]

this last expression can be placed in $\sum_{j=0}^{k-1}\frac{\soooop}{(t^\nu\lambda)^{2j}(t\lambda)^{2(k-j)}}.$ This takes care of $H_1(v_{2k-1}).$ The contributions of the other three terms can be analyzed just as in \cite{KST}, but we reproduce the calculations for completeness, starting with $N_{2k-1}(v_{2k-1}).$ First note that adding up the $v_j$ we get

\begin{equation}\label{ukdifference}
u_{2k-2}-u_0\in\frac{1}{(t\lambda)^2}IS^1(R\log R,Q).
\end{equation}

For this note that extra powers of $1/(t\lambda)^2$ (respectively $1/(t^\nu\lambda)^2$) can be written as powers of $b_4$ (respectively $b_5),$ and as such can safely be placed in $IS^0(1).$ Also note that from lemma \ref{lemma3.8} in the appendix, if

\[v\in\frac{1}{(t\lambda)^2}IS^3(R\log R,\mathcal{Q}),\]

then

\[\sin v\in\frac{1}{(t\lambda)^2}IS^3(R\log R,Q),\quad \cos v\in IS^0(1,\mathcal{Q}).\]

Using this and (\ref{ukdifference}) we compute

\begin{align*}
\cos(2u_0)-\cos(2u_{2k-2})&=2\cos(2u_0)\sin^2(u_{2k-2}-u_0)\\
                           &~~+2\sin(2u_0)\sin(u_{2k-2}-u_0)\cos(u_{2k-2}-u_0)\\
                           &\in\frac{1}{(t\lambda)^4}IS^6(R^2(\log R)^2,\mathcal{Q})\\
                           &~~+\frac{1}{(t\lambda)^2}IS^4(\log R,\mathcal{Q}).
\end{align*}

Hence,

\begin{align*}
t^2&\frac{\cos(2u_0)-\cos(2u_{2k-2})}{\sin^2\alpha}v_{2k-1}=\frac{\alpha^2}{\sin^2\alpha}\frac{(t\lambda)^2(\cos(2u_0)-\cos(2u_{2k-2}))}{R^2}v_{2k-1}\\
              &\in\frac{(t\lambda)^2}{R^2}\Big(\frac{1}{(t\lambda)^2}IS^4(\log R,\mathcal{Q})+\frac{1}{(t\lambda)^4}IS^6(R^2(\log R)^2,\mathcal{Q})\Big)\\
              &~~\times\sum_{j=0}^{k-1}\frac{1}{(t^\nu\lambda)^{2j}(t\lambda)^{2(k-j)}}\stooo\\
              &\subseteq \sum_{j=0}^{k-1}\frac{\Big(IS^5(R^{-1}(\log R)^{2k},\mathcal{Q})+\frac{1}{(t\lambda)^2}IS^7(R(\log R)^{2k+1},\mathcal{Q})\Big)}{(t^\nu\lambda)^{2j}(t\lambda)^{2(k-j)}}\\
              &\subseteq \sum_{j=0}^{k-1}\frac{1}{(t^\nu\lambda)^{2j}(t\lambda)^{2(k-j)}}IS^5(R(\log R)^{2k-1},\mathcal{Q})\\
              &\subseteq \sum_{j=0}^{k-1}\frac{1}{(t^\nu\lambda)^{2j}(t\lambda)^{2(k-j)}}\soooop
\end{align*}

Here in the step before the last we have pulled out a factor of $b^2$ from the second factor, and given away a factor of $R^2$ to gain a factor of $\log R$ in the first factor. For the next term in $N_{2k-1}(v_{2k-1}),$ note also that

\[1-\cos(2v_{2k-1})=v^2_{2k-1}g(v^2_{2k-2})\]

for some analytic function $g.$ Moreover writing $1/(t\lambda)^2=b_4$ and $1/(t^\nu\lambda)^2=b_5,$ we can think of $v_{2k-1}$ as belonging to $\stooo$ whenever we do not need the extra time decay. Therefore since $g$ has a convergent Taylor series at $0,~ g(v_{2k-2})\in S^0(1,\mathcal{Q}),$ and hence

\[1-\cos(2v_{2k-1})\in\Big(\sum_{j=0}^{k-1}\frac{1}{(t^\nu\lambda)^{2j}(t\lambda)^{2(k-j)}}\stooo\Big)^2.\]

Using this observation we compute

\begin{align*}
t^2&\frac{\sin(2u_{2k-2})}{2\sin^2\alpha}(1-\cos(2v_{2k-1}))\\
                              &\in\frac{\alpha^2}{\sin^2\alpha}\frac{(t\lambda)^2}{R^2}\Big(IS^1(R^{-1},\mathcal{Q})+\frac{1}{(t\lambda)^2}IS^3(R\log R,\mathcal{Q})\Big)\\
                              &~~\times\Big(\sum_{j=0}^{k-1}\frac{1}{(t^\nu\lambda)^{2j}(t\lambda)^{2(k-j)}}\stooo\Big)^2\\
                              &\subseteq \sum_{j=0}^{k-1}\frac{1}{(t^\nu\lambda)^{2j}(t\lambda)^{2(k-j)}}\Big(\sum_{i=0}^{k-1}\frac{1}{(t^\nu\lambda)^{2i}(t\lambda)^{2(k-1-i)}}IS^5(R^{-1}(\log R)^{4k-2},\mathcal{Q})\\
                              &\quad~~+\sum_{i=0}^{k-1}\frac{1}{(t^\nu\lambda)^{2i}(t\lambda)^{2(k-i)}}IS^7(R(\log R)^{4k-1},\mathcal{Q})\Big)\\
                              &\subseteq \sum_{j=0}^{k-1}\frac{1}{(t^\nu\lambda)^{2j}(t\lambda)^{2(k-j)}} IS^5(R (\log R)^{2k-1},\mathcal{Q}).
\end{align*}

We have again pulled out factors of $b_1^lb_2^m,~m+l=k,$ to pass to the last inclusion. Finally, using a similar observation as above,

\begin{align*}
t^2\frac{\cos(2u_{2k-2})}{\sin^2\alpha}&(2v_{2k-1}-\sin(2v_{2k-1}))\\
                                       &\in\frac{(t\lambda)^2}{R^2}IS^0(1,\mathcal{Q})\Big(\sum_{j=0}^{k-1}\frac{IS^3(R(\log R)^{2k-1},\mathcal{Q})}{(t^\nu\lambda)^{2j}(t\lambda)^{2(k-j)}}\Big)^3\\
                                       &\subseteq \sum_{j,l,m=0}^{k-1}\frac{1}{(t^\nu\lambda)^{2(j+l+m)}(t\lambda)^{2(3k-j-l-m)-2}}IS^7(R(\log R)^{6k-3},\mathcal{Q})\\
                                       &\subseteq \sum_{j=0}^{k-1}\frac{1}{(t^\nu\lambda)^{2j}(t\lambda)^{2(k-j)}}IS^7(R(\log R)^{2k-1},\mathcal{Q}).
\end{align*}

This concludes the analysis of $N_{2k-1}(v_{2k-1}).$ For $E^tv_{2k-1}$ we can ignore the dependence on $a$ and therefore the observation that

\[t^2\partial^2_t\Big(\sum_{j=0}^{k-1}\frac{IS^3(R(\log R)^{2k-1})}{(t^\nu\lambda)^{2j}(t\lambda)^{2(k-j)}}\Big)\subseteq \sum_{j=0}^{k-1}\frac{IS^1(R(\log R)^{2k-1})}{(t^\nu\lambda)^{2j}(t\lambda)^{2(k-j)}},\]

shows that the contribution of this term can be placed in the right space. Finally for $E^av_{2k-1}$ we use the representation (\ref{voddaltrep}) again, together with the fact that

\begin{equation}\label{QQpoperators}
a\partial_a,a^{-1}\partial_a,(1-a^2)\partial^2_a:\mathcal{Q}\rightarrow \mathcal{Q}^\p.
\end{equation}

The contribution of $\frac{W_j(R,a)}{(t^\nu\lambda)^{2j}(t\lambda)^{2(k-j)}}$ to $t^2E^av_{2k-1}$ is

\begin{align*}
\frac{1}{(t^\nu\lambda)^{2j}(t\lambda)^{2(k-j)}}\big[&2(2j+2\nu(k-j)a\partial_aW_j)-2(\nu+1)Ra\partial^2_{Ra}W_j\\
                                                     &+2Ra^{-1}\partial^2_{Ra}W_j+a^{-1}\partial_aW_j\\
                                                     &+(1-a^2)\partial^2_aW_j-2a\partial_aW_j\big]
\end{align*}

In view of (\ref{QQpoperators}), if we sum the expression above over $j$ we see that

\begin{align*}
t^2E^av_{2k-1}\in\sum_{j=0}^{k-1}\frac{1}{(t^\nu\lambda)^{2j}(t\lambda)^{2(k-j)}}\soooop,
\end{align*}

as desired.\\


Step $3:$ define $v_{2k}$ such that (\ref{veven}) holds.\\

Write $e_{2k-1}=\sum_{l=0}^{k-1}e_{2k-1,l}$ with $t^2e_{2k-1,l}\in\frac{\soooop}{(t^\nu\lambda)^{2l}(t\lambda)^{2(k-l)}}.$ Setting $b_i=0$ we write the asymptotic component $f_{2k-1,l}$ of the principal part of $e_{2k-1,l}$ near $R=\infty$ as

\begin{align*}
t^2f_{2k-1,l}&=\frac{R}{(t^\nu\lambda)^{2l}(t\lambda)^{2(k-l)}}\sum_{j=0}^{2k-1}q_j(a)(\log R)^j\\
              &=\frac{1}{(t^\nu\lambda)^{2l}(t\lambda)^{2(k-l)-1}}\sum_{j=0}^{2k-1}aq_j(a)(\log R)^j,
\end{align*}

with $q_j\in\mathcal{Q}^\p.$ In view of (\ref{veveneq}) we consider

\begin{align*}
t^2(-\partial_t^2+\partial_\alpha^2+\frac{1}{\alpha}\partial_\alpha-\frac{1}{\alpha^2})w_{2k,l}=t^2f_{2k-1},
\end{align*}

and homogeneity considerations lead us to seek solutions of the form

\begin{align*}
w_{2k,l}=\frac{1}{(t^\nu\lambda)^{2l}(t\lambda)^{2(k-l)-1}}\sum_{j=0}^{2k-1}W^j_{2k,l}(a)(\log R)^j.
\end{align*}

Matching the powers of $\log R$ we see that $W^j_{2k,l}$ has to satisfy

\begin{align*}
t^2\Big(-\partial_t^2+\partial_\alpha^2+\frac{1}{\alpha}\partial_\alpha-\frac{1}{\alpha^2}\Big)\Big(\frac{W^j_{2k,l}}{(t^\nu\lambda)^{2l}(t\lambda)^{2(k-l)-1}}\Big)=\frac{aq_j(q)-F_j(a)}{(t^\nu\lambda)^{2l}(t\lambda)^{2(k-l)-1}},
\end{align*}

where with $\beta=\beta(k,l)=2l+\nu(2(k-l)-1),$

\begin{align}
F_j(a)=&(j+1)\big[((1+\nu)\beta+a^{-2})W^{j+1}_{2k,l}\nonumber\\
       &~~~~~~~~~~+2(a^{-1}-(1+\nu)a)\partial_aW^{j+1}_{2k,l}\big]\nonumber\\
       &(j+2)(j+1)(a^{-2}-(1+\nu)^2)W^{j+2}_{2k,l}. \label{aFj}
\end{align}

We are using the convention $W^j_{2k,l}=0$ for $j\geq 2k$ here, and will solve this system successively in $j$ as $j$ decreases from $2k-1$ to $0.$ Conjugating the power of $t$ we get

\begin{align*}
t^2\Big(-(\partial_t+\frac{\beta}{t})^2+\partial_\alpha^2+\frac{1}{\alpha}\partial_\alpha-\frac{1}{\alpha^2}\Big)W^j_{2k,l}=aq_j(a)-F_j(a).
\end{align*}

With

\begin{equation}\label{lbeta}
L_{\beta}:=(1-a^2)\partial_a^2+(a^{-1}+2a\beta-2a)\partial_a+(-\beta^2+\beta-a^{-2}),
\end{equation}

we rewrite this last equation as

\begin{equation}\label{aequation}
L_\beta W^j_{2k,l}=aq_j(a)-F_j(a).
\end{equation}

Our goal is to show that with zero Cauchy data imposed at $a=0,$

\begin{equation}\label{aimprovement}
W^j_{2k,l}\in a^3 \mathcal{Q},\quad j=0,1,\dots, 2k-1,\quad l=0,1,\dots,k-1.
\end{equation}

For this we will use lemma \ref{lemma 3.9} in the appendix. In particular, we use the fact that for $\beta>1/2$ the unique solution to the equation

\begin{align*}
L_\beta w=f,\quad w(0)=0, \quad \partial_a w(0)=0,
\end{align*}

has the form

\begin{align}
w(a)=c_1\phi_1(a)+c_2\phi_2(a)&+c_3\phi_1(a)\int_a^1\phi_2(a^\p)q_1(a^\p)f(a^\p)da^\p\nonumber\\
                           &+c_4\phi_2(a)\int_{a_0}^a\phi_1(a^\p)q_1(a^\p)f(a^\p)da^\p. \label{asolutionform}
\end{align}

Here $a_0<1$ is some number close to $1,$ the $c_i$ are constants which may depend on $\beta,$ and the fundamental solutions $\phi_i$ can be written as

\begin{equation}\label{aphirational}
\phi_1(a)=r_1(a),\quad \phi_2(a)=r_2(a)(1-a)^{\beta+\frac{1}{2}},
\end{equation}

with $r_1,~r_2$ analytic, if $\beta-\frac{1}{2}\notin \ZZ^{+},$ and $\phi_1$ has to be modified to

\begin{equation}\label{aphiirrational}
\phi_1(a)=r_1(a)+c\phi_2(a)\log(1-a),
\end{equation}

for some constant $c$ (depending on $\beta$) if $\beta-\frac{1}{2}\in \ZZ^{+}.$ The function $q_1$ appearing in the particular solution can be written as

\begin{equation}\label{aqone}
q_1(a)=(1-a)^{-\beta-\frac{1}{2}}r_3(a),
\end{equation}

with $r_3$ analytic near $a=1.$

The following integral identity, valid for $\gamma\neq -1$ and positive integers $m,$ will also be useful

\begin{align}\label{aintegralid}
\int_a^1(1-a^\p)^\gamma&[\log(1-a^\p)]^mda^\p\nonumber\\
                       &=\frac{1}{\gamma+1}[\log(1-a)]^m(1-a)^{\gamma+1}\nonumber\\
                       &~~-\frac{m}{\gamma+1}\int_a^1(1-a^\p)^{\gamma}[\log(1-a^\p)]^{m-1}da^\p.
\end{align}

We will prove (\ref{aimprovement}) inductively on $k.$ We will write $c$ for generic constants, which may depend on $\beta$ and $\nu.$ Note that near $a=1,$ proving (\ref{aimprovement}) for fixed $l$ is tantamount to finding an absolutely convergent expansion

\begin{align}
W^j_{2k,l}(a)=q_0(a)+\sum_{i=1}^\infty&\Big((1-a)^{(2i-1)\nu+\frac{1}{2}}\sum_{j=0}^{\infty}q_{2i-1,j}(a)(\log(1-a))^j\nonumber\\
                                      &+(1-a)^{2i\nu+1}\sum_{j=0}^{\infty}q_{2i,j}(a)(\log(1-a))^j\Big),\label{aimprovementone}
\end{align}

such that $q_0$ and $q_{ij}$ satisfy the conditions of definition \ref{Qdef}. To start off the induction, note that when $k=1,$  $l=0$ so we are in the same setting as in \cite{KST}, and (\ref{aequation}) becomes

\begin{align*}
&L_\nu W^1_{2,0}(a)=ah(a),\\
&L_\nu W^0_{2,0}(a)=(c-a^{-2})W^1_{2,0}(a)+(ca-a^{-1})\partial_a W^1_{2,0}(a),
\end{align*}

where $h$ is analytic with an even expansion at $a=0.$ The behavior near $a=0$ is a direct consequence of part (i) of lemma \ref{lemma 3.9}. For $a$ close to $1$ we denote by $h_i$ and $g_i$ generic functions that are analytic near $a=1,$ and use (\ref{asolutionform})-(\ref{aintegralid}) to write

\begin{align*}
&W^1_{2,0}(a)=g_0(a)+g_1(a)(1-a)^{\nu+\frac{1}{2}},\\
&W^0_{2,0}(a)=h_0(a)+h_1(a)(1-a)^{\nu+\frac{1}{2}} +h_2(a)(1-a)^{\nu+\frac{1}{2}}\log(1-a),
\end{align*}

if $\nu-\frac{1}{2}\notin\ZZ^{+},$ and

\begin{align*}
&W^1_{2,0}(a)=g_0(a)+g_1(a)(1-a)^{\nu+\frac{1}{2}} +g_2(a)(1-a)^{\nu+\frac{1}{2}}\log(1-a),\\
&W^0_{2,0}=h_0(a)+(1-a)^{\nu+\frac{1}{2}}\sum_{m=0}^2h_{m+1}(a)[\log(1-a)]^m\\
&\quad\quad\quad     +(1-a)^{2\nu+1}\sum_{m=0}^2h_{m+4}(a)[\log(1-a)]^m,
\end{align*}

if $\nu-\frac{1}{2}\in\ZZ^{+}.$ This finishes the first step of the induction. Now we assume we have proved (\ref{aimprovement}) up to $k-1$ and prove it for $k.$ Here we are assuming that once we establish (\ref{aimprovement}) we can prove (\ref{veven}) and (\ref{eeven}) as well, and therefore we assume that $q_j\in\mathcal{Q}^\p$ in (\ref{aequation}). That this can be done, is shown below in the remainder of step $3$ and in step $4.$ \footnote{This is not circular; in fact one can view this step as part of the larger iteration scheme (\ref{vodd})-(\ref{eeven}).}

We fix $l\leq k-1,$ and write $\beta$ for

\begin{equation}\label{beta}
\beta(k,l)=2l+(2(k-l)-1)\nu.
\end{equation}

We will generically write $g_j,~h_j,~g_{i,j}$ or $h_{i,j}$ for functions that are analytic near $a=1.$ The first equation we need to solve is

\begin{align*}
L_\beta W^{2k-1}_{2k,l}=aq_{2k-1}(a).
\end{align*}

Again the behavior near $a=0$ is a consequence of part (i) of lemma \ref{lemma 3.9}. For the behavior bear $a=1,$ note that the solution is given by the right hand side of (\ref{asolutionform}) with $f(a)=aq_{2k-1}(a).$ The contribution of $c_1\phi_1+c_2\phi_2$ can be written as

\[g_0(a)+g_1(a)(1-a)^{\beta+\frac{1}{2}}+g_2(a)(1-a)^{\beta+\frac{1}{2}}\log(1-a).\]

Note that if $l=0$ then

\[(1-a)^{\beta+\frac{1}{2}} =(1-a)^{(2k-1)\nu+\frac{1}{2}}\]

and otherwise

\[(1-a)^{\beta+\frac{1}{2}}=(1-a)^{(2m+1)\nu+\frac{1}{2}}g_4(a),\]

where $g_4(1)=0,$ and $m$ is an integer. This shows that the contribution of $c_1\phi_1+c_2\phi_2$ has the right form. Noting that $\phi_2q_1$ is analytic near $a=1,$ we can write the contribution of $\phi_1\int_a^1\phi_2q_1fda^\p$ as

\begin{align*}
\phi_1\Bigg(g_0(a)+\sum_{i=1}^\infty&\Big((1-a)^{(2i-1)\nu+\frac{1}{2}}\sum_{j=0}^{\infty}g_{2i-1,j}(a)(\log(1-a))^j\\
                                         &+(1-a)^{2i\nu+1}\sum_{j=0}^{\infty}g_{2i,j}(a)(\log(1-a))^j\Big)\Bigg).
\end{align*}

Regardless of whether $\nu$ is rational, this can be written as

\begin{align*}
g_0(a)+\sum_{i=1}^\infty&\Big((1-a)^{(2i-1)\nu+\frac{1}{2}}\sum_{j=0}^{\infty}g_{2i-1,j}(a)(\log(1-a))^j\\
                                         &+(1-a)^{2i\nu+1}\sum_{j=0}^{\infty}g_{2i,j}(a)(\log(1-a))^j\\
                                         &+(1-a)^{(2i-1)\nu+\beta+1}\sum_{j=0}^{\infty}h_{2i,j}(a)(\log(1-a))^j\\
                                         &+(1-a)^{2i\nu+\beta+\frac{3}{2}}\sum_{j=0}^{\infty}h_{2i-1,j}(a)(\log(1-a))^j\Big).
\end{align*}

According to (\ref{beta}),

 \[(1-a)^{(2i-1)\nu+\beta+1}=(1-a)^{2m\nu+1}\]

for some integer $m\geq i$ (depending on $i$), and

 \[(1-a)^{2i\nu+\beta+\frac{3}{2}}=(1-a)^{(2n-1)\nu+\frac{1}{2}}(1-a)\]

for some integer $n\geq i+1$ (depending on $i$), and therefore the expression above can be included in the right hand side of (\ref{aimprovementone}) (that for fixed $i$ the sums over $j$ are finite follows from the corresponding fact for $f$). We consider the contribution of $\phi_2\int_{a_0}^a\phi_1q_1fda^\p$ next. Here we may get terms of the form

\[c(1-a)^{-1}[\log(1-a)]^m\]

in the integrand, which cannot be dealt with using (\ref{aintegralid}), and which contribute extra logarithms. To make this more precise, write

\[\phi_1(a)=g_0(a)+g_1(a)(1-a)^{\beta+\frac{1}{2}}\log(1-a).\]

First we consider the contribution of the analytic part $g_0.$ This gives integrals of the forms

\begin{align*}
&h_0(a)(1-a)^{\beta+\frac{1}{2}}\int_{a_0}^ag_{2i-1,j}(a^\p)(1-a^\p)^{(2i-1)\nu-\beta-1}(\log(1-a^\p))^jda^\p,\quad \mathrm{or}\\
&h_0(a)(1-a)^{\beta+\frac{1}{2}}\int_{a_0}^ag_{2i,j}(a^\p)(1-a^\p)^{2i\nu-\beta-\frac{1}{2}}(\log(1-a^\p))^jda^\p.
\end{align*}

Both of these have the right form, and only add an extra power of $\log (1-a)$ if $(2i-1)\nu-\beta-1$ or $2i\nu-\beta-\frac{1}{2}$ are negative integers. For the contribution of $g_1(a)(1-a)^{\beta+\frac{1}{2}}\log(1-a),$ note that

\[q_1(a)g_1(a)(1-a)^{\beta+\frac{1}{2}}\log(1-a)=h_1(a)\log(1-a),\]

so we need to consider integrals of the form

\begin{align*}
&h_0(a)(1-a)^{\beta+\frac{1}{2}}\int_{a_0}^ag_{2i-1,j}(a^\p)(1-a^\p)^{(2i-1)\nu-\frac{1}{2}}(\log(1-a^\p))^{j+1}da^\p,\quad \mathrm{or}\\
&h_0(a)(1-a)^{\beta+\frac{1}{2}}\int_{a_0}^ag_{2i,j}(a^\p)(1-a^\p)^{2i\nu}(\log(1-a^\p))^{j+1}da^\p,
\end{align*}

which again have the right form. This completes the case of $j={2k-1}.$

For $j<2k-1,$ we can assume by induction that $W^{j+1}_{2k,l}$ and $W^{j+2}_{2k,l}$ have the right form, and in particular they have a cubic Taylor expansion at $a=0.$ It follows from (\ref{aFj}) that $aq_j(a)-F_j(a)$ has a Taylor expansion at $a=0$ beginning with a linear term. The behavior of $W^{j}_{2k,l}$ at $a=0$ is therefore a consequence of part (i) of lemma \ref{lemma 3.9}. Moreover, $\mathcal{Q}\subseteq\mathcal{Q}^\p$ implies that $aq_j(a)-F_j(a)\in\mathcal{Q}^\p,$ and the expansion of $W^{j}_{2k,l}$ follows in exactly the same way as above. This finishes the proof of (\ref{aimprovement}).

We cannot use $w_{2k,l}$ for $v_{2k,l},$ because $\log R$ is singular at $R=0.$ However, in view of the computations above, we define

\[v_{2k,l}:\frac{1}{(t^\nu\lambda)^{2l}(t\lambda)^{2(k-l)-1}}\sum_{j=0}^{2k-1}W^j_{2k,l}(a)\Big(\frac{1}{2}\log (1+R^2)\Big)^j,\]

and analyze the error near $R=0$ generated by this modification in the next step. Pulling out a factor of $a^3=R^3/(t\lambda)^3$ we see that

\[v_{2k}:=\sum_{j=0}^{k-1}v_{2k,j}\in \sum_{j=0}^{k-1}\frac{1}{(t^\nu\lambda)^{2j}(t\lambda)^{2(k+1-j)}}\sttoe.\]
Step $4:$ show that with $v_{2k}$ defined as in the previous step, (\ref{eeven}) holds.\\

We write

\begin{align*}
t^2e_{2k}=&t^2(e_{2k-1}-e_{2k-1}^0)\\
          &+t^2(e^0_{2k-1}-(-\partial_t^2+\partial_\alpha^2+\frac{1}{\alpha}\partial_\alpha-\frac{1}{\alpha^2})v_{2k})\\
          &+t^2N_{2k}(v_{2k})+t^2H_0(v_{2k}),
\end{align*}

with $N_{2k}$ and $H_0$ as in (\ref{Neven}) and (\ref{Heven}) respectively, and

\[t^2e^0_{2k-1}=\sum_{l=0}^{k-1}\frac{R}{(t^\nu\lambda)^{2l}(t\lambda)^{2(k-l)}}\sum_{j=0}^{2k-1}q_{l,j}(a)\Big(\frac{1}{2}\log(1+ R^2)\Big)^j.\]

The first term has the form

\begin{align*}
t^2(e_{2k-1}-e^0_{2k-1})\in\sum_{l=0}^{k-1}\frac{1}{(t^\nu\lambda)^{2l}(t\lambda)^{2(k-l)}}\Big[&IS^1(R^{-1}(\log R)^{2k},\mathcal{Q}^\p)\\
                                                                                               &+\sum_{i=1}^3b_iIS^1(R(\log R)^{2k-1},\mathcal{Q}^\p)\Big].
\end{align*}

The last sum is contained (\ref{eeven}). For the first term we can write, with $w\in IS^1(R^{-1}(\log R)^{2k},\mathcal{Q}^\p),$

\[w=(1-a^2)w+\frac{R^2w}{(t\lambda)^2}\in IS^1(R^{-1}(\log R)^{2k}, \mathcal{Q})+ b_1IS^1(R(\log R)^{2k-1},\mathcal{Q}^\p),\]

and therefore the first sum can be placed in (\ref{eeven}) too. According to the computations in step 3, the second term in $e_{2k}$ would be zero if we had $\log R$ instead of $\frac{1}{2}\log(1+R^2)$ in both $e^0_{2k}$ and $v_{2k}.$ The difference is therefore obtained, if we replace the derivatives of $\frac{1}{2}\log(1+R^2)$ by derivatives of $\log R$ in

\begin{align*}
&t^2(\partial_t^2+\partial_\alpha^2+\frac{1}{\alpha}\partial_\alpha)v_{2k}\\
&=t^2\Big(\partial_t^2+\partial_\alpha^2+\frac{1}{\alpha}\partial_\alpha\Big)\Big(\sum_{l=0}^{k-1}\frac{1}{(t^\nu\lambda)^{2l}(t\lambda)^{2(k-l)-1}}\times\\
&\quad\quad\quad\quad\quad\quad\quad\quad\quad\quad\quad\quad\quad \sum_{j=0}^{2k-1}W_{2k,l}^j(a)\Big(\frac{1}{2}\log(1+R^2)\Big)^j\Big).
\end{align*}

This difference can be written as a sum of expressions of the form

\begin{align*}
\frac{1}{(t^\nu\lambda)^{2l}(t\lambda)^{2(k-l)-1}}\sum_{j=0}^{2k-1}\Bigg(\frac{W^j_{2k,l}(a)}{a^2}\Big[&S^0(R^{-2})(\log(1+R^2))^{j-1}\\
                                                                                               &+S^0(R^{-2})(\log(1+R^2))^{j-2}\Big]\\
                                                                                               &+\frac{\partial_aW^j_{2k,l}(a)}{a}S^0(R^{-2})(\log(1+R^2))^{j-1}\Bigg).
\end{align*}

Since $W^{j}_{2k,l}$ is cubic at $0,$ we can pull out a factor of $a$ and see that this belongs to

\[\frac{1}{(t^\nu\lambda)^{2l}(t\lambda)^{2(k-l)}}IS^1(R^{-1}(\log(1+R^2))^{2k-2},\mathcal{Q}^\p),\]

which is admissible as we saw above. We consider $H_0(v_{2k})$ next. We write

\[f(\alpha)=\frac{\alpha^2-\sin^2\alpha}{\alpha^2\sin^2\alpha},\]

and note that $f$ is an analytic function of $b_3.$ We also write $W_{2k,l}^j(a)=a^3G^j_{2k,l}(a)$ with $G^j_{2k,l}\in\mathcal{Q}.$ There are two terms to consider. The first is

\begin{align*}
t^2(\frac{\alpha^2-\sin^2\alpha}{\alpha^2\sin^2\alpha})v_{2k}&=f(\alpha)\sum_{l=0}^{k-1}\sum_{j=0}^{2k-1}\frac{t^2a^3G^j_{2k,l}(a)}{(t^\nu\lambda)^{2l}(t\lambda)^{2(k-l)-1}}\Big(\frac{1}{2}\log(1+R^2)\Big)^j\\
                                                             &=f(\alpha)\sum_{l=0}^{k-1}\sum_{j=0}^{2k-1}\frac{Rb_3G^j_{2k,l}(a)}{(t^\nu\lambda)^{2l}(t\lambda)^{2(k-l)}}\Big(\frac{1}{2}\log(1+R^2)\Big)^j\\
                                                             &\in \sum_{l=0}^{k-1}\frac{1}{(t^\nu\lambda)^{2l}(t\lambda)^{2(k-l)}}b_3\soooep,
\end{align*}

as desired. The other term is similar. Write

\[g(\alpha)=\frac{\sin\alpha-\alpha\cos\alpha}{\alpha^2\sin\alpha},\]

and note that $g$ is an analytic function of $b_3.$ The remaining term in $H_0$ is

\begin{align*}
t^2g(\alpha)\alpha\partial_\alpha v_{2k}.
\end{align*}

We expand $v_{2k}$ as before. If $\alpha\partial_\alpha$ hits $G^j_{2k,l},$ the fact that $a\partial_a$ maps $\mathcal{Q}$ to $\mathcal{Q}^\p$ allows us to estimate the resulting term as above. Similarly $\alpha\partial_\alpha(a^3)=3a^3,$ so the contribution can be included in (\ref{eeven}) if $\alpha\partial_\alpha$ hits $a^3.$ It remains to consider

\[\alpha\partial_\alpha\Big((\frac{1}{2}\log(1+R^2))\Big)^j=\frac{jR^2}{1+R^2}\Big(\frac{1}{2}\log(1+R^2)\Big)^{j-1},\]

which is even better than what we already had. This completes the analysis of $H_0(v_{2k}).$ Finally, we consider the nonlinearity $N_{2k}(v_{2k}).$ This is dealt with exactly as in \cite{KST}, but we reproduce the calculations for completeness. Summing up the $v_j$ we have

\[u_{2k-1}-u_0\in\frac{1}{(t\lambda)^2}IS^3(R\log R,\mathcal{Q}).\]

We will write $f(\alpha)$ for a generic even analytic function of $\alpha.$ Using lemma \ref{lemma3.8} in the appendix we get

\begin{align*}
t^2\frac{1-\cos(2u_{2k-1})}{\sin^2\alpha}v_{2k}&=f(\alpha)\frac{(t\lambda)^2}{R^2}(1-\cos(2u_{2k-1}))v_{2k}\\
                                               &\in\frac{(t\lambda)^2}{R^2}\Big(IS^1(R^{-1},\mathcal{Q})+\frac{1}{(t\lambda)^2}IS^3(R\log R,\mathcal{Q})\Big)^2\\
                                               &~~\times\sum_{j=0}^{k-1}\frac{1}{(t^\nu\lambda)^{2j}(t\lambda)^{2(k+1-j)}}IS^3(R^3(\log R)^{2k-1},\mathcal{Q})\\
                                               &\subseteq \sum_{j=0}^{k-1}\frac{1}{(t^\nu\lambda)^{2j}(t\lambda)^{2(k-j)}}\Big(IS^3(R^{-1}(\log R)^{2k-1},\mathcal{Q})\\
                                               &\quad\quad+\frac{1}{(t\lambda)^{2}}IS^5(R(\log R)^{2k},\mathcal{Q})\\
                                               &\quad\quad+\frac{1}{(t\lambda)^4}IS^7(R^3(\log R)^{2k-1},\mathcal{Q})\Big)\\
                                               &\subseteq\sum_{j=0}^{k-1}\frac{1}{(t^\nu\lambda)^{2j}(t\lambda)^{2(k-j)}}\Big(IS^3(R^{-1}(\log R)^{2k-1},\mathcal{Q})\\
                                               &\quad\quad\quad+\frac{b_1}{(t\lambda)^2}IS^5(R(\log R)^{2k-1},\mathcal{Q})\Big).
\end{align*}

For the quadratic term we have

\begin{align*}
t^2\frac{\sin(2u_{2k-1})}{2\sin^2\alpha}&(1-\cos(2v_{2k}))\\
                                        &\in\frac{(t\lambda)^2}{R^2}\Big(IS^1(R^{-1},\mathcal{Q})+\frac{1}{(t\lambda)^2}IS^3(R\log R,\mathcal{Q})\Big)\\
                                        &~~\times\Big(\sum_{j=0}^{k-1}\frac{1}{(t^\nu\lambda)^{2j}(t\lambda)^{2(k+1-j)}}IS^3(R^3(\log R)^{2k-1},\mathcal{Q})\Big)^2\\
                                        &\subseteq \sum_{j=0}^{k-1}\frac{1}{(t^\nu\lambda)^{2j}(t\lambda)^{2(k-j)}}\Big(\sum_{l=0}^{k-1}\frac{IS^5(R^3(\log R)^{4k-2},\mathcal{Q})}{(t^\nu\lambda)^{2l}(t\lambda)^{2(k+1-l)}}\\
                                        &\quad\quad\quad+\sum_{l=0}^{k-1}\frac{IS^7(R^5(\log R)^{4k-1},\mathcal{Q})}{(t^\nu\lambda)^{2l}(t\lambda)^{2(k+2-l)}}\Big)\\
                                        &\subseteq \sum_{j=0}^{k-1}\frac{1}{(t^\nu\lambda)^{2j}(t\lambda)^{2(k-j)}}\big(IS^1(R^{-1}(\log R)^{2k},\mathcal{Q})\\
                                        &\quad\quad\quad+\sum_{i=1}^2b_iIS^3(R(\log R)^{2k-1},\mathcal{Q})\big).
\end{align*}

Finally, the cubic term can be bounded as

\begin{align*}
t^2\frac{\cos(2u_{2k-1})}{\sin^2\alpha}&(2v_{2k}-\sin(2v_{2k}))\\
                                       &\in\frac{(t\lambda)^2}{R^2}\Big(\sum_{j=0}^{k-1}\frac{IS^3(R^3(\log R)^{2k-1},\mathcal{Q})}{(t^\nu\lambda)^{2j}(t\lambda)^{2(k-j)}}\Big)^3\\
                                       &\subseteq \sum_{j,l,m=0}^{k-1}\frac{IS^7(R^7(\log R)^{6k-3},\mathcal{Q})}{(t^\nu\lambda)^{2(j+m+l)}(t\lambda)^{2(3k+2-j-l-m)}}\\
                                       &\subseteq\sum_{j=0}^{k-1}\sum_{i=0}^{4k-2}\frac{a^6b_1^ib_2^{4k-2-i}}{(t^\nu\lambda)^{2j}(t\lambda)^{2(k-j)}}IS^1(R(\log R)^{2k-1},\mathcal{Q})\\
                                       &\subseteq \sum_{i=1}^2\sum_{j=0}^{k-1}\frac{b_i}{(t^\nu\lambda)^{2j}(t\lambda)^{2(k-j)}}IS^1(R(\log R)^{2k-1},\mathcal{Q}^\p).
\end{align*}

This finishes the proof of Theorem \ref{elliptic modifier}.
\end{proof}
\section{The Perturbed Equation}
We will now start the construction of a complete solution from the approximate solution of the previous section. We let

\[u(t,\alpha)=u_{2k-1}(t,\alpha)+\varepsilon(t,\alpha).\]

This section is devoted to deriving the equation $\varepsilon$ has to satisfy for $u$ to be a solution, and recasting this equation in a new coordinate system which allows easier analysis of the problem. In the $\alpha-t$ coordinates $\varepsilon$ has to satisfy

\begin{align*}
-\varepsilon_{tt}+\varepsilon_{\alpha\alpha}+\cot\alpha\varepsilon_\alpha-\frac{\cos(2Q(\lambda\alpha))}{\sin^2\alpha}\varepsilon=N_{2k-1}(\varepsilon)+e_{2k-1},
\end{align*}

or

\begin{equation}\label{epsilontalpha}
-\varepsilon_{tt}+\varepsilon_{\alpha\alpha}+\frac{1}{\alpha}\varepsilon_\alpha-\frac{\cos(2Q(\lambda\alpha))}{\alpha^2}\varepsilon=N_{2k-1}(\varepsilon)+H_1(\varepsilon)+e_{2k-1}.
\end{equation}

Here $e_{2k-1}$ is the error at step $2k-1$ from the previous section. $N_{2k-1}$ is defined in (\ref{Nodd}), but with $u_{2k}$ replaced by $u_{2k-1},$ and $H_1$ is given by $(\ref{Hodd}).$  We change variables to $R=\lambda(t)\alpha$ and $\tau=\frac{-1}{\nu}t^{-\nu}$ (so $\frac{\partial t}{\partial\tau}=\frac{1}{\lambda}),$ and let $v(\tau,R):=\varepsilon(t(\tau),\lambda^{-1}R).$ Inserting this into (\ref{epsilontalpha}) and rearranging we get

\begin{align*}
&-\Big[\Big(\partial_\tau+\frac{\lambda_\tau}{\lambda}R\partial_R\Big)^2+\frac{\lambda_\tau}{\lambda}\Big(\partial_\tau+\frac{\lambda}{\lambda_\tau}R\partial_R\Big)\Big]v\\
&~\quad\quad\quad\quad\quad+\Big(\partial_R^2+\frac{1}{R}\partial_R-\frac{\cos(2Q(R))}{R^2}\Big)v\\
&\quad\quad\quad\quad\quad\quad\quad=\frac{1}{\lambda^2}[N_{2k-1}(\varepsilon)+H_1(\varepsilon)+e_{2k-1}](t(\tau),\lambda^{-1}R).
\end{align*}

Our next goal is to transform $\partial_R^2+\frac{1}{R}\partial_R-\frac{\cos(2Q(R))}{R^2}$ into a self adjoint operator on $L^2(\RR^{+},dR).$ To accomplish this we let $\ep(\tau,R):=R^{\frac{1}{2}}v(\tau,R),$ and rewrite the equation above in terms of $\ep$ as

\begin{align}
\Big(-\Big(\partial_\tau+\frac{\lambda_\tau}{\lambda}R\partial_R\Big)^2&+\frac{1}{4}(\frac{\lambda_\tau}{\lambda})^2+\frac{1}{2}\partial_\tau(\frac{\lambda_\tau}{\lambda})\Big)\ep-\LL\ep\nonumber\\
&=\lambda^{-2}R^{\frac{1}{2}}(N_{2k-1}(R^{-\frac{1}{2}}\ep)+H_1(R^{-\frac{1}{2}}\ep)+e_{2k-1}), \label{epeq}
\end{align}

where

\begin{equation}\label{Ldef}
\LL:=-\partial_R^2+\frac{3}{4R^2}-\frac{8}{(1+R^2)^2}.
\end{equation}

(\ref{epeq}) is the main equation we have to solve in this paper. $\LL$ is self adjoint with respect to $L^2(\RR^{+},dR)$ (see the appendix for the definition of the domain of $\LL$) and its spectral properties are summarized in lemma \ref{spectrallemma} of the appendix. We will adopt the notation of lemma \ref{spectrallemma} in the remainder of this paper. In particular recall the definition of the distorted Fourier transform associated with $\LL$

\[\FF:f\mapsto \hat{f}(\xi)=\lim_{b\rightarrow\infty}\int_0^b\phi(R,\xi)f(R)dR,\]

with inverse

\[\FF^{-1}:\hat{f}\mapsto f(r)=\lim_{\mu\rightarrow\infty}\int_0^\mu\phi(R,\xi)\hat{f}(\xi)\rho(\xi)d\xi.\]

Here $\rho(\xi)d\xi$ is the spectral measure of $\LL.$ The idea is to expand $\ep$ in terms of the generalized Fourier basis $\phi(R,\xi)$ as

\[\ep(\tau,R)=\int_0^\infty x(\tau,R)\phi(R,\xi)\rho(\xi)d\xi,\]

and derive a transport like equation for the coefficient $x(\tau,\xi)$ from (\ref{epeq}). The operator $R\partial_R$ in (\ref{aequation}) is not diagonal in the Fourier basis, but we can use the construction in \cite{KST} to deal with the error operator $\KK$ defined by

\begin{equation}\label{K}
\widehat{R\partial_R u}=-2\xi\partial_\xi \hat{u}+ \KK\hat{u}.
\end{equation}

With $\KK$ defined as such, we have

\[\FF\Big(\partial_\tau+\frac{\lambda_\tau}{\lambda}R\partial_R\Big)=\Big(\partial_\tau+\frac{\lambda_\tau}{\lambda}(-2\xi\partial_\xi+\KK)\Big)\FF,\]

which yields

\begin{align*}
\FF\Big(\partial_\tau+\frac{\lambda_\tau}{\lambda}R\partial_R\Big)^2&=\Big(\partial_\tau+\frac{\lambda_\tau}{\lambda}(-2\xi\partial_\xi+\KK)\Big)^2\FF\\
                                                                    &=\Big(\partial_\tau-\frac{\lambda_\tau}{\lambda}2\xi\partial_\xi\Big)^2\FF+2\frac{\lambda_\tau}{\lambda}\KK\Big(\partial_\tau-\frac{\lambda_\tau}{\lambda}2\xi\partial_\xi\Big)\FF\\
                                                                    &~+\frac{\lambda_\tau^2}{\lambda^2}\Big(\KK^2-\frac{\nu}{1+\nu}\KK+2[\KK,\xi\partial_\xi]\Big)\FF.
\end{align*}

We can now apply our Fourier transform to (\ref{epeq}) to get (with $x=\mathcal{F}\ep$)

\begin{align}
-\Big(\partial_\tau-2\frac{\lambda_\tau}{\lambda}\xi\partial_\xi\Big)^2&x-\xi x\nonumber\\
                                                                    &=2\frac{\lambda_\tau}{\lambda}\KK\Big(\partial_\tau-2\frac{\lambda_\tau}{\lambda}\xi\partial_\xi\Big)x+\frac{\lambda_\tau^2}{\lambda^2}\Big(\KK^2-\frac{\nu}{1+\nu}\KK+2[\KK,\xi\partial_\xi]\Big)x\nonumber\\
                                                                    &~~-\Big(\frac{1}{4}\Big(\frac{\lambda_\tau}{\lambda}\Big)^2+\frac{1}{2}\partial_\tau\Big(\frac{\lambda_\tau}{\lambda}\Big)\Big)x\nonumber\\
                                                                    &~~+\lambda^{-2}\FF R^{\frac{1}{2}}\Big(N_{2k-1}(R^{-\frac{1}{2}}\FF^{-1}x)+H_1(R^{-\frac{1}{2}}\FF^{-1}x)+e_{2k-1}\Big).\label{xequation}
\end{align}

Our strategy will be to solve this equation for the Fourier coefficient $x$ and use the inverse Fourier transform $\mathcal{F}^{-1}$ to retrieve a solution of (\ref{epeq}). Note that we are interested in solutions of (\ref{xequation}) which decay as $\tau\rightarrow\infty.$ This means that we have to solve the equation backwards in time with zero Cauchy data at $\tau=\infty.$ The next section is devoted to describing the mapping properties of the error operator $\KK$ and the fundamental solution of the "transport like" equation

\begin{equation}\label{transport}
-\Bigg[\Big(\partial_\tau-2\frac{\lambda_\tau}{\lambda}\xi\partial_\xi\Big)^2+\xi\Bigg] x(\tau,\xi)=b(\tau,\xi).
\end{equation}

We call this a "transport like" equation because the characteristic curves of the operator $\partial_\tau-2\frac{\lambda_\tau}{\lambda}\xi\partial_\xi$ are $(\tau,\lambda^{-2}(\tau)\xi)$ (see section 8 of \cite{KST} for more details).

\section{The Fundamental Solution and Estimates in $H^s_\rho$}
In this section we define certain Sobolev spaces associated with the operator $\LL.$ We will also provide the mapping properties of the fundamental solution operator of (\ref{transport}) and the operator $\KK$ relative to these spaces. Finally we demonstrate how to control the norm of the right hand side of (\ref{xequation}). In the next section we will use a contraction mapping argument on these Sobolev spaces to find a solution $\ep$ to (\ref{epeq}) with the right decay properties. The estimates from this section are the necessary ingredients for being able to close that fixed point argument.\\

We begin with the norms. On the frequency side we define the weighted $L^2$ norms

\begin{equation}\label{Ls}
\|f\|_{L^{2,s}_\rho}:=\Big(\int_0^\infty|f(\xi)|^2\langle\xi\rangle^{2s}\rho(\xi)d\xi\Big)^{\frac{1}{2}}.
\end{equation}

For functions of the spacial variable $R$ we define

\begin{align*}
\|u\|_{H^s_\rho}:=\|\hat{u}\|_{L^{2,s}_\rho}.
\end{align*}

The relationship between this norm and the usual Sobolev norm on $\RR^2$ is explained in lemma \ref{lemma10.1} in the appendix. Finally to control the decay in time we introduce the $L^{N,\infty}L^{2,s}_\rho$ spaces with norm

\begin{align*}
\|f\|_{L^{\infty,N}L^{2,s}_\rho}:=\sup_{\tau\geq1}\tau^N\|f\|_{L^{2,s}_\rho}.
\end{align*}

$L^{N,\infty}H^s_\rho$ is defined similarly.\\

We now move to the estimates. A priori we only have

\[\KK:C^\infty_0((0,\infty))\rightarrow C^\infty((0,\infty)).\]

However, with the norms just defined, according to lemma \ref{transference} in the appendix, $\KK$ has the following mapping properties

\begin{align}
&\KK:L^{2,s}_\rho\rightarrow L^{2,s+\frac{1}{2}}_\rho, \label{Kmap}\\
&[\KK,\xi\partial_\xi]:L^{2,s}_\rho\rightarrow L^{2,s}_\rho.\label{Kcommap}
\end{align}

Next let $H$ be the backward fundamental solution of the operator

\begin{equation}\label{fundop}
\Big(\partial_\tau-2\frac{\lambda_\tau}{\lambda}\xi\partial_\xi\Big)^2+\xi,
\end{equation}

and denote its operator kernel by $H(\tau,\sigma).$ By this we mean that suppressing the $\xi$ variable,

\[x(\tau)=-\int_\tau^\infty H(\tau,\sigma)b(\sigma)\]

is a solution of (\ref{transport}). Then according to lemma \ref{fundamentallemma} in the appendix, given $s$ there is constant $C=C(s)$ such that for large enough $N$

\begin{equation}\label{fundamental}
\|Hb\|_{L^{\infty,N-2}L^{2,s+\frac{1}{2}}_\rho}+\Big{\|}\big(\partial_\tau-2\frac{\lambda_\tau}{\lambda}\xi\partial_\xi)Hb\Big{\|}_{L^{\infty,N-1}L^{2,s}_\rho}\leq CN^{-1} \|b\|_{L^{\infty,N}L^{2,s}_\rho}.
\end{equation}

Motivated by this we want to bound the $L^{\infty,N}L^{2,s}_\rho$ norm of the right hand side of (\ref{xequation}) by the $L^{\infty,N-2}L^{2,s+\frac{1}{2}}_\rho$ norm of $x$ and the $L^{\infty,N-1}L_{\rho}^{2,s}$ norm of $(\partial_\tau-2\frac{\lambda_\tau}{\lambda}\xi\partial_\xi)x.$ The terms involving $\KK$ or its commutator can be bounded using (\ref{Kmap}) and (\ref{Kcommap}). The term coming from $e_{2k-1}$ will be dealt with in the next section. Note that this term does not depend on $x$ so we only need to make sure that it belongs to our iteration space (to be specified in the next section). This will be accomplished by taking $k$ large as in \cite{KST}. For the $N_{2k-1}$ term we will use lemma \ref{nonlinearity} from the appendix. Therefore the term with $H_1$ is the only one which is essentially new. The estimate we need is

\begin{equation}\label{Hbound}
\|\lambda^{-2}R^{\frac{1}{2}}H_1(R^{-\frac{1}{2}}\ep)\|_{L^{\infty,N}H^s_\rho}\lesssim\|\ep\|_{L^{\infty,N-2}H^{s+\frac{1}{2}}_\rho}.
\end{equation}

Recall that

\[H_1(v)=\frac{\sin\alpha -\alpha\cos\alpha}{\alpha\sin\alpha}\partial_\alpha v+\frac{\sin^2\alpha-\alpha^2}{\alpha^2\sin^2\alpha}v.\]

We write

\[z(\alpha)=\frac{\sin\alpha-\alpha\cos\alpha}{\alpha^2\sin\alpha},\]

so

\begin{align*}
\lambda^{-2}R^{\frac{1}{2}}H_1(R^{-\frac{1}{2}}\ep)=&\lambda^{-2}\Big(\frac{\sin^2\alpha-\alpha^2}{\alpha^2\sin^2\alpha}\Big)\ep+z(\alpha)\lambda^{-2}\alpha R^{\frac{1}{2}}\lambda\partial_R(R^{-\frac{1}{2}}\ep).
\end{align*}

Since $\lambda^{-2}=c\tau^{-2-\frac{2}{\nu}},$ and according to lemma \ref{lemma9.1}, the first term above can be bounded consistently with (\ref{Hbound}). For the second term note that as long as $\nu\leq1$ we have $\lambda^{-2}\lambda=t^{-\nu+1+2\nu}=c\tau^{-2}t^{1-\nu}\lesssim\tau^{-2},$ which is he desired time gain in (\ref{Hbound}). Note moreover, that since we are only concerned with finding a solution inside the light cone, we may replace $r$ by $q(\frac{R}{\lambda})\frac{R}{\lambda}$ where $q$ is a smooth function supported in $R\leq1.$ Therefore, again revoking lemma \ref{lemma9.1}, for such range of $\nu$ it suffices to bound

\[\|\frac{R^{\frac{1}{2}}q(\frac{R}{\lambda})}{\lambda}R\partial_R(R^{-\frac{1}{2}}\ep)\|_{H^s_\rho}\lesssim\|\ep\|_{H^{2+\frac{1}{2}}_\rho}.\]

But noting that in two dimensions, $R\partial_R=x\partial_x+y\partial_y,$ lemma \ref{lemma10.1} allows us to reduce this to

\[\|\frac{q}{\lambda}(x\partial_x+y\partial_y)f\|_{H^{2s}(\RR^2)}\lesssim\|f\|_{H^{2s+1}(\RR^2)},\]

where $f=R^{-\frac{1}{2}}\ep.$ This holds true because because $\frac{qx}{\lambda}$ and $\frac{qy}{\lambda},$ together with all of their derivative, are bounded (see e.g. \cite{Ta1}).
\section{The Blow Up Construction: Proof of Theorem \ref{BlowUp}}
In this section we combine the results from the previous sections to find $\ep$ in a way that leads to blow up for $u.$ Fix $\nu\in(\frac{1}{2},1]$ and let $u_{2k-1}$ and $e_{2k-1}$ be as in theorem \ref{elliptic modifier}. The index $k$ is chosen sufficiently large depending $\nu$ and in particular on $N$ (as will be specified below). A priori $u_{2k-1}$ and $e_{2k-1}$ are defined only on the cone $\{\alpha\leq t\},$ but we extend them to functions on the double cone $\{\alpha\leq2t\},$ with the only requirement that extensions are of the same smoothness class and all meaningful derivatives agree on the boundary of the cone. With these choices of $\nu,$ $u_{2k-1}$ and $e_{2k-1,}$ and with $s\in(\frac{1}{4},\frac{\nu}{2})$ we will find a solution $\ep$ of (\ref{epeq}) (solved backwards in $\tau$) which satisfies

\begin{equation}\label{epbounds}
\|\ep(\tau)\|_{H^{s+\frac{1}{2}}_\rho}\lesssim\tau^{2-N}, \quad\Big{\|}\Big(\partial_\tau+\frac{\lambda_\tau}{\lambda}R\partial_R\Big)\ep(\tau)\Big{\|}_{H^s_\rho}\lesssim\tau^{1-N}.
\end{equation}

The large exponent $N$ is determined by lemmas \ref{fundamentallemma} and \ref{nonlinearity}. In what follows we identify the Sobolev spaces on the domain sphere (more precisely on a neighborhood of the north pole) with the normal Sobolev spaces on $\RR^2$ via the $(\alpha,\theta)$ (or $(R,\theta)$) coordinates. \\\\

Near the boundary of the cone $\{\alpha=t\},$ the error $e_{2k-1}$ has a singularity of the type $(1-a)^{\nu-\frac{1}{2}}\log^m(1-a),$ which means that locally near the boundary $e_{2k-1}\in H^\beta$ as long as $\beta<\nu.$ Note that under the transformation $T$ of lemma \ref{lemma10.1} this corresponds to $H^{\beta/2}_\rho.$ On the other hand according to theorem \ref{elliptic modifier}

\[t^2e_{2k-1}\in\sum_{j=0}^{k-1}\frac{1}{(t^\nu\lambda)^{2j}(t\lambda)^{2(k-j)}}\soooop.\]

Near $R=0$ this means that we have an expansion of the form (with $T$ as in lemma \ref{lemma10.1})

\[T(R^{\frac{1}{2}}e_{2k-1})=e^{i\theta}R(c_0(\tau)+c_1(\tau)R^2+c_2(\tau)R^4+\cdots),\]

which is smooth around $R=0.$ Finally we consider the size of the error for large $R$

\[e_{2k-1}=\bigo\Big(\frac{R(\log(2+R))^{2k-1}}{t^{-2k+2}t^2(t\lambda)^2}\Big).\]

It follows that the $L^2$ norm of $T(\lambda^{-2}R^{\frac{1}{2}}e_{2k-1})$ is bounded by $C\tau^2t^{2k-2}$ (note that we are taking the $L^2$ norm only on the light cone which corresponds to $R\lesssim\tau$). Taking $k$ so large that

\[\sup_{\tau\geq1}\tau^{N+2}t^{2k-2}<\infty,\]

and noting that due to the expansion of $e_{2k-1}$ for large $R$ taking $R$ derivatives reduces the size for large $R,$ the arguments above show that

\begin{equation}\label{errorcontrol}
\|\lambda^{-2}R^{\frac{1}{2}}e_{2k-1}\|_{L^{\infty,N}H^s_\rho}\lesssim1,
\end{equation}

as long as $s<\frac{\nu}{2}.$ Now using the distorted Fourier transform $\mathcal{F}$ we recast (\ref{epeq}) as (\ref{xequation}) with $x=\mathcal{F}\ep.$ According to lemmas \ref{transference}, \ref{fundamentallemma}, \ref{nonlinearity}, \ref{lemma9.1}, and (\ref{Hbound}) and (\ref{errorcontrol}) we can solve (\ref{xequation}) using a contraction mapping argument with respect to the norm

\[\|x\|_{L^{\infty,N-2}L^{2,s+\frac{1}{2}}_\rho}+\Big{\|}\Big(\partial_\tau-2\frac{\lambda_\tau}{\lambda}\xi\partial_\xi\Big)x\Big{\|}_{L^{\infty,N-1}L^{2,s}_\rho}.\]

Using the inverse of the distorted Fourier transform and lemma \ref{transference} we can derive a solution $\ep$ of (\ref{epeq}) which satisfies (\ref{epbounds}).\\\\

To go from the longitudinal variable $u$ to a full co-rotational wave map in terms of the ambient coordinates of $\RR^3\supseteq S^2,$ observe that these coordinates are given by $\phi\circ T(u),$ where $\phi:\RR^2\rightarrow S^2\subseteq\RR^3$ is given by

\begin{align*}
\phi(e^{i\theta}v)=(\cos v,\sin v \cos\theta,\sin v \sin\theta).
\end{align*}

It is then verified that $\phi\circ T(u)\in H^{2s+1},$ interpreted componentwise. We have thus constructed a wave map on the closure of the cone $\{\alpha\leq t,~0<t\leq t_0\},$ which is of class $H^{1+\nu-}.$ To extend this to a wave map on all of $S^2\times \RR,$ extend the initial data $\partial_tu(t_0,\cdot),~u(t_0,\cdot)$ at time $t=t_0$ to all of $S^2$ in the same smoothness and equivariance class, and in such a way that a large neighborhood of the south pole ($\alpha=\pi$) is mapped to the south pole on the target. Let $\tilde{u}$ be the corresponding wave map. We claim that $\tilde{u}$ does not develop singularities in $(0,t_0]\times S^2.$ Indeed due to the equivariance assumption the first singularity has to develop at one of the poles. At the north pole $\alpha=0$ this is precluded by our construction because finite speed of propagation implies that $\tilde{u}$ agrees with $u$ on the light cone $\{\alpha\leq t,~0<t\leq t_0\}.$ At the south pole this is precluded by the small energy result in \cite{ST} which implies that singularity development results in energy concentration, which in turn is ruled out by the constancy of the initial data near this pole and finite speed of propagation. This completes the proof of theorem \ref{BlowUp}. 
\section{Appendix}
%
%
In this section we provide the statements, without proofs, of some of the results from \cite{KST} which are used in the current work.

\begin{lemma}\label{lemma3.7}
Let $k\geq 1,$ and let $L$ be defined as in (\ref{Lelliptic}). Then the solution $v$ to the equation
\[Lv=f\in S^1(R^{-1}(\log R)^{2k-2}),\quad v(0)=v^\p(0)=0\]
has regularity
\[v\in S^3(R(\log R)^{2k-1}).\]
\end{lemma}
\begin{proof}
Lemma 3.7, p. 560 in \cite{KST}.
\end{proof}

\begin{lemma}\label{lemma3.8}
Let

\[v\in\frac{1}{(t\lambda)^2}IS^3(R\log R,Q).\]

Then

\[\sin v\in\frac{1}{(t\lambda)^2}IS^3(R\log R,Q),\quad \cos v\in IS^0(1,Q).\]

\end{lemma}

\begin{proof}
Lemma 3.8, pp. 561-562 in \cite{KST}.
\end{proof}

\begin{lemma}\label{lemma 3.9}
Let $f$ be an analytic function in $[0,1)$ with an odd expansion at $0,$ and let $\beta>\frac{1}{2}$ be a fixed real number. Then there is a unique solution $w$ to the equation

\[L_\beta w=f,\quad w(0)=A,\quad \partial_a w(0)=B,\]

with

\[L_\beta = (1-a)^2\partial^2_a+(a^{-1}+2a\beta-2a)\partial_a+(-\beta^2+\beta-a^2),\]

such that

\begin{description}

\item[] (i) if $A=B=0,$ then w is analytic in $[0,1)$ with an odd cubic expansion at 0.\\
\item[] (ii) \begin{align*}
             w(a)= &c_1\phi_1(a)+c_2\phi_2(a)\\
             &-(\beta+1/2)^{-1}\phi_1(a)\int_a^1\phi_2(a^\p)q_1(a^\p)f(a^\p)da^\p\\
             &-(\beta+1/2)^{-1}\phi_2(a)\int_{a_0}^a\phi_1(a^\p)q_1(a^\p)f(a^\p)da^\p.
             \end{align*}
\end{description}
Here $a_0<1$ is some number close to one, and $c_1$ and $c_2$ are constants determined by the initial values. $q_1$ is given by

\[q_1(a)=(1-a)^{-\beta-\frac{1}{2}}[1+(1-a)\tilde{q_2}(a)]\]

with $\tilde{q_2}$ analytic near $a=1.$ Moreover we can find constant $\mu_l$ and $\tilde{\mu_l}$ such that

\[\phi_1(a)=1+\sum_{l=1}^\infty\mu_l(1-a)^l,\quad \phi_2(a)=(1-a)^{\beta+\frac{1}{2}}\Big[1+\sum_{l=1}^\infty\tilde{\mu_l}(1-a)^l\Big],\]

if $\beta-\frac{1}{2}\notin \ZZ^{+}.$ If $\beta-\frac{1}{2}\in\ZZ^{+}$ then we need to modify $\phi_1$ to

\[\phi_1(a)=1+\sum_{l=1}^\infty\mu_l(1-a)^l+c\phi_2(a)\log(1-a),\]

for some constant $c$ depending on $\beta.$
\end{lemma}

\begin{proof}
All statements come from lemma 3.9 in \cite{KST}. Part (i) corresponds to part (i) there.  Part (ii) comes from equation (3.31) in the proof of lemma 3.9 in \cite{KST}. The expressions for $\phi_j$ are given by (3.27) and (3.30), and for $q_1$ at the bottom of page 567, in \cite{KST}.
\end{proof}

\begin{lemma}\label{spectrallemma}
Let $\LL$ be the self adjoint operator $-\partial_R^2+\frac{3}{4R^2}-\frac{8}{(1+R^2)^2}$ on $L^2((0,\infty),dR)$ with domain

\[\mathrm{Dom}(\LL)=\{f\in L^2((0,\infty))|f,f^\p\in AC_{loc}((0.\infty)),\LL_0f\in L^2((0,\infty))\},\]

where $\LL_0:=-\partial_R^2+\frac{2}{4R^2}.$ Then:

\begin{description}
\item[]  $~~$ (i) The spectrum of $\LL$ is purely absolutely continuous and equals $spec(\LL)=[0,\infty).$\\

\item[]  $~~$ (ii) For each $z\in\CC$ there exists a fundamental system $\phi(R,z),~\theta(R,z)$ for $\LL-z$ which is analytic in $z$ for each $R>0$ and has the asymptotic behavior

    \[\phi(R,z)\sim R^{\frac{3}{2}},\quad \theta(R,z)\sim \frac{1}{2}R^{-\frac{1}{2}}\quad \mathrm{as}~R\rightarrow 0.\]

    In particular, their Wronksian is $W(\theta(.,z),\phi(.,z))=1$ for all $z\in\CC.$ By convention $\theta(R,z)$ and $\phi(R,z)$ are real valued for $z\in\RR.$ $\phi$ is the Weyl-Titchmarsh solution of $\LL-z$ at $R=0,$ and admits an absolutely convergent asymptotic expansion

    \[\phi(R,z)=\phi_0(R)+R^{-\frac{1}{2}}\sum_{j=1}^{\infty}(R^2z)^j\phi_j(R^2)\]

    where

    \[\phi_0(R)=\frac{R^{\frac{3}{2}}}{1+R^2},\]

    and for $j\geq 1$ the functions $\phi_j$ are holomorphic in $U=\{\mathrm{Re}~u>-\frac{1}{2}\}$ and satisfy the bounds

    \[|\phi_j(u)|\leq\frac{3C^j}{(j-1)!}\log(1+|u|).\]
    In particular $\phi_j(0)=0$ and $|\phi^\p_j(0)|\leq\frac{3C^j}{(j-1)!}$ for all $j\geq1.$\\

\item[]  $~~$ (iii) For each $z\in \CC,$ $\mathrm{Im}~z>0,$ let $\psi^+(R,z)$ denote the Weyl-Titchmarsh solution of $\LL-z$ at $R=\infty$ normalized so that

    \[\psi^+(R,z)\sim z^{-\frac{1}{4}}e^{iz^{\frac{1}{2}}R}\quad\mathrm{as}~R\rightarrow \infty,\quad\mathrm{Im}~z^{\frac{1}{2}} >0.\]

      If $\xi>0,$ then the limit $\psi^+(R,\xi+i0)$ exists point-wise for all $R>0$ and we denote it by $\psi^+(R,\xi).$ Moreover, define $\psi^-(.,\xi):=\overline{\psi^+(.,\xi)}.$ Then $\psi^+(R,\xi),~\psi^-(R,\xi)$ form a fundamental system of $\LL-\xi$ with asymptotic behavior $\psi^{\pm}(R,\xi)\sim \xi^{-\frac{1}{4}}e^{\pm i\xi^{\frac{1}{2}}R}$ as $R\rightarrow\infty.$ More precisely, for and $\xi>0$ we can write

      \[\psi^+(R,\xi)=\xi^{-\frac{1}{4}}e^{iR\xi^{\frac{1}{2}}}\sigma(R\xi^{\frac{1}{2}},R),\quad R^2\xi\gtrsim1\]

      where $\sigma$ admits the asymptotic series approximation

      \begin{align*}
      \sigma(q,R)\approx \sum_{j=0}^{\infty}q^{-j}\psi^+_j(R),\quad \psi_0^+=1
      \end{align*}

      with zero order symbols $\psi^+_j$ that are analytic at infinity,

      \[\sup_{R>0}|(R\partial_R)^k\psi^+_j(R)|<\infty\]

      in the sense that for all large integers $j_0,$ and all indices $\alpha,\beta$ we have

      \[\sup_{R>0}\Big|(R\partial_R)^\alpha(q\partial_q)^\beta\Big[\sigma(q,R)-\sum_{j=0}^{j_0}q^{-j}\psi_j^+(R)\Big]\Big|\leq c_{\alpha,\beta,j_0}q^{-j_0-1}\]

      for all $q>1.$\\

\item[]  $~~$ (iv) The spectral measure $\rho(\xi)d\xi$ of $\LL$ is absolutely continuous and its density is given by

\[\rho(\xi)=\frac{1}{\pi}\mathrm{Im}~m(\xi+i0)\chi_{[\xi>0]}\]

with the "generalized Weyl-Titchmarsh" function

\[m(z)=\frac{W(\theta(.,z),\psi^+(.,z))}{W(\psi^+(.,z),\phi(.,z))},\quad\mathrm{Im}~z\geq0.\]

Moreover $\rho$ satisfies \footnote{$a\asymp b$ means that for some constant $C$ one has $C^{-1}a<b<Ca.$}

\begin{align*}
\rho(\xi)\asymp
\left\{ \begin{array}{l}
\frac{1}{\xi(\log \xi)^2} \quad\quad\quad~ \xi\ll1\\
\xi\quad\quad\quad\quad\quad\quad  \xi\gtrsim 1,
\end{array} \right.
\end{align*}

and we can write

\begin{equation}\label{rhorep}
\rho(\xi)=\frac{1}{\pi}|a(\xi)|^{-2}
\end{equation}

where $a$ satisfies the symbol type bounds

\[|(\xi\partial_\xi)^ka(\xi)|\leq c_k|a(\xi)|\quad \forall~\xi>0.\]
\item[]  $~~$ (v) The distorted Fourier transform defined as

\[\FF:f\mapsto \hat{f}(\xi)=\lim_{b\rightarrow\infty}\int_0^b\phi(R,\xi)f(R)dR\]

is a unitary operator from $L^2(\RR^+)$ to $L^2(\RR^+,\rho),$ and its inverse is given by

\[\FF^{-1}:\hat{f}\mapsto f(r)=\lim_{\mu\rightarrow\infty}\int_0^\mu\phi(R,\xi)\hat{f}(\xi)\rho(\xi)d\xi.\]

Here $\lim$ refers to the $L^2(\RR^+,\rho),$ respectively the $L^2(\RR^+),$ limit.
\end{description}
\end{lemma}

\begin{proof}
Unless otherwise stated, "lemma x," "proposition x" or "theorem x" refer to the corresponding lemma, proposition or theorem in \cite{KST}.  Part (i) is lemma 5.2. Part (ii) is part (a) of theorem 5.3. For the asymptotic behavior of $\phi$ see lemma 5.4. Part (iii) comes from part (b) of theorem 5.3 and proposition 5.6. (iv) comes from (c) of theorem 5.3 and proposition 5.7. (v) is (d) of theorem 5.3. Also see \cite{G}, and in particular section 3 and example 3.10 there, for a general analysis of Schrodinger operators with a singular potential similar to $\LL.$
\end{proof}

\begin{lemma}\label{transference}
The operator $\KK$ from (\ref{K}) maps

\[\KK:L^{2,s}_\rho\rightarrow L^{2,s+\frac{1}{2}}_\rho.\]

In addition we have the commutator bound

\[[\KK,\xi\partial_\xi]:L^{2,s}_\rho\rightarrow L^{2,s}_\rho.\]

Both statements hold for all $s \in \RR.$
\end{lemma}

\begin{proof}
This is the content of proposition 6.2 in \cite{KST}. See also theorem 6.1 there.
\end{proof}

\begin{lemma}\label{lemma10.1}
Define the map

\[u(R)\mapsto Tu(R,\theta)=e^{i\theta}R^{-\frac{1}{2}}u(R),\]

where the right hand side is interpreted as a function in $\RR^2$ expressed in polar coordinates $(R,\theta).$ Then $T$ is an isometry

\[T:L^2(\RR^+)\rightarrow L^2(\RR^2),\]

and for any $s\geq0$ we have

\[\|u\|_{H^{s/2}_\rho}(\RR^+)\asymp \|Tu\|_{H^s(\RR^2)}\]

in the sense that if one side is finite then the other is finite and they have comparable sizes.
\end{lemma}

\begin{proof}
See lemma 10.1 in \cite{KST}.
\end{proof}

\begin{lemma}\label{fundamentallemma}
Let $H$ be the fundamental solution of the operator (\ref{fundop}). Given $s\geq0$ let $N$ be large enough. Then there is a constant $C$ which depends on $s$ but not on $N$ so that

\begin{align*}
\|Hb\|_{L^{\infty,N-2}L^{2,s+\frac{1}{2}}_\rho}+\Big{\|}\big(\partial_\tau-2\frac{\lambda_\tau}{\lambda}\xi\partial_\xi)Hb\Big{\|}_{L^{\infty,N-1}L^{2,s}_\rho}\leq CN^{-1} \|b\|_{L^{\infty,N}L^{2,s}_\rho}.
\end{align*}

\end{lemma}

\begin{proof}
See proposition 7.1 and corollary 7.2 in \cite{KST}.
\end{proof}

\begin{lemma}\label{nonlinearity}
Assume that $N$ is large enough and that $\frac{1}{4}<s<\frac{3}{4}+\frac{\nu}{2}.$ Then the map

\[x\mapsto \lambda^{-2}\FF\big(R^{\frac{1}{2}}N_{2k-1}(R^{-\frac{1}{2}}\FF^{-1}x)\big)\]

is locally Lipschitz from $L^{\infty,N-2}L^{2,s+\frac{1}{2}}_\rho$ to $L^{\infty,N}L^{2,s}_\rho.$

\end{lemma}

\begin{proof}
Proposition 7.3 in \cite{KST}.
\end{proof}

\begin{lemma}\label{lemma9.1}
Let $q\in S(1,\mathcal{Q})$ and $|s|<\frac{\nu}{2}+\frac{3}{4}.$ Then

\[\|qf\|_{H^s_\rho}\lesssim\|f\|_{H^s_\rho}.\]

\end{lemma}

\begin{proof}
This is lemma 9.1 in \cite{KST}. Note that the algebras in our paper are slightly different from those in \cite{KST}. In particular there is an extra dependency on the variables $b_2,\cdots,b_5.$ This extra dependency is irrelevant in the proof of the lemma and can be ignored just as the $b_1$ variable is ignored in the proof of the lemma in \cite{KST}.
\end{proof}

\end{document}